\documentclass[11pt, a4paper]{article}

\usepackage{amsmath}\usepackage{amsfonts}\usepackage{amssymb}\usepackage{amsthm}
\usepackage{graphicx}
\usepackage{ifthen}
\usepackage{lscape}
\usepackage{subfigure}
\usepackage{setspace}
\usepackage{subeqnarray}
\usepackage{url}
\usepackage{xspace}

\newenvironment{keywords}{\footnotesize{\bf Keywords: }}{}
\newenvironment{AMS}{\footnotesize{\bf AMS subject classification: }}{}
\usepackage[left=1in,top=1.5in,right=.97in,bottom=1in]{geometry}

\usepackage{verbatim}
\usepackage{mdwlist}
\usepackage{xcolor}
\usepackage{hyperref}
\usepackage{ulem}
\allowdisplaybreaks

\newtheorem{theorem}{Theorem}[section]

\newtheorem{remark}{Remark}[section]
\numberwithin{equation}{section}
\numberwithin{figure}{section}

\long\def\symbolfootnote[#1]#2{\begingroup\def\thefootnote{\fnsymbol{footnote}}
\footnote[#1]{#2}\endgroup}

\renewcommand{\paragraph}[1]{}
\renewcommand{\paragraph}[1]{}
\usepackage{verbatim}
\usepackage{ifpdf}
\ifpdf
\else
\renewcommand{\includegraphics}[1]{\framebox{Graphics Placeholder}}
\renewcommand{\includegraphics}[2][1]{\framebox{Graphics Placeholder}}
\fi

\definecolor{changecol}{rgb}{0.7, 0, 0}
\newcommand{\change}[1]{{\color{changecol}{#1}}}
\renewcommand{\change}[1]{{#1}}
\definecolor{aacol}{rgb}{0.7, 0, 0}
\definecolor{ascol}{rgb}{0, 0, 0.7}
\definecolor{plcol}{rgb}{0, 0.5, 0}
\definecolor{refcol}{rgb}{0.2, 0.2, 0.2}

\author{Assyr Abdulle\thanks{Section of Mathematics, Swiss Federal Institute of Technology (EPFL), Station 8, CH-1015, Lausanne, Switzerland
        }
        \and
		Ping Lin\thanks{Division of Mathematics,
        University of Dundee, 23 Perth Road, Dundee, Scotland DD1 4HN, UK
        }
        \and
        Alexander V. Shapeev$^{\dagger,}$\thanks{Present address: 
        School of Mathematics, 206 Church St. SE, University of Minnesota, Minneapolis,
        MN 55455, US}
        }

\title{Numerical Methods for Multilattices\footnote{Published in {\it Multiscale Model. Simul.}, {\bf 10}(3): 696--726, 2012.}}

\newcommand{\smfrac}[2]{{\textstyle\frac{#1}{#2}}}
\newcommand{\eps}{{\epsilon}}
\newcommand{\bbR}{{\mathbb R}}
\newcommand{\bbZ}{{\mathbb Z}}
\newcommand{\bbN}{{\mathbb N}}
\newcommand{\calA}{{\mathcal A}}

\newcommand{\calL}{{\mathcal L}}
\newcommand{\calM}{{\mathcal M}}
\newcommand{\calU}{{\mathcal U}}
\newcommand{\calQ}{{\mathcal Q}}
\newcommand{\calR}{{\mathcal R}}

\newcommand{\calT}{{\mathcal T}}
\newcommand{\calP}{{\mathcal P}}

\newcommand{\bq}{{\mathbf q}}

\newcommand{\mF}{{\sf F}}
\newcommand{\mG}{{\sf G}}

\newcommand{\calTh}{{\mathcal T}_h}

\def\del{\delta\hspace{-1pt}}
\def\ddel{\delta^2\hspace{-1pt}}

\newcommand{\delE}{{\del E}}
\newcommand{\ddelE}{{\ddel E}}
\newcommand{\delPhi}{{\del\Phi}}
\newcommand{\deltildeE}{{\del_{u^h}\!\tildeE}}
\newcommand{\delR}{{\del R}}
\newcommand{\delPi}{{\delta\hspace{-0.5pt}\Pi}}

\newcommand{\qc}{{\rm qc}}
\newcommand{\mqc}{{\rm mqc}}

\newcommand{\hqc}{{\rm hqc}}

\renewcommand{\c}{{\rm c}}
\newcommand{\per}{{\rm per}}
\newcommand{\rep}{{\rm rep}}
\newcommand{\lin}{{\rm lin}}

\newcommand{\tildeE}{\tilde E}
\newcommand{\baru}{\bar u}
\def\<{\langle}
\def\>{\rangle}
\newcommand{\dd}{{\rm d}}
\newcommand{\dx}{\dd x}

\begin{document}
\sloppy
\maketitle

\begin{abstract}
Among the efficient numerical methods based on atomistic models, the quasicontinuum (QC) method has attracted growing interest in recent years.
The QC method was first developed for crystalline materials with Bravais lattice and was later extended to multilattices (Tadmor et al, 1999).
Another existing numerical approach to modeling multilattices is homogenization.
In the present paper we review the existing numerical methods for multilattices and propose another concurrent macro-to-micro method in the numerical homogenization framework.
We give a unified mathematical formulation of the new and the existing methods and show their equivalence.
We then consider extensions of the proposed method to time-dependent problems and to random materials.
\end{abstract}



\pagestyle{myheadings}
\thispagestyle{plain}

\section{Introduction}

In some applications of solid mechanics, such as modeling cracks, structural defects, or nanoelectromechanical systems, the classical continuum description is not suitable, and one is required to utilize an atomistic description of materials.
However, full atomistic simulations are prohibitively expensive, hence one needs to coarse-grain the problem.
The quasicontinuum (QC) method \cite{TadmorPhillipsOrtiz1996} is one of the most efficient methods of coarse-graining the atomistic statics.
The idea behind QC is to introduce piecewise affine constraints for the atoms in regions with smooth deformation and use the Cauchy--Born rule to define the energy of the corresponding groups of constrained atoms.
To formulate the QC method for multilattice crystals one must account for relative shifts of Bravais lattices of which the multilattice is comprised \cite{TadmorSmithBernsteinEtAl1999}.

The QC method is a multiscale method capable of coupling atomistic and continuum description of materials.
It is intended to model an atomistic material in a continuum manner in the regions where the deformation is smooth and use the fully atomistic model only in the small neighborhood of defects, thus effectively reducing the degrees of freedom of the system.
Originally, the QC method was developed for crystalline materials with a (single) Bravais lattice \cite{TadmorPhillipsOrtiz1996}, and the convergence of a few variants of the method has been analyzed under some practical assumptions (see, e.g., \cite{DobsonLuskinOrtner2010b, VanKotenLiLuskinEtAl2012, Lin2003, Lin2007, LuMing, MingYang2009, OrtnerShapeev2011, OrtnerSuli2008}).
The QC method is based on the so-called Cauchy--Born rule (see, e.g., \cite{BlancLeBrisLions2007a, EMing2007, Ericksen2008, FrieseckeTheil2002}) which states that the energy of a certain volume of a material can be approximated through the deformation energy density, which is computed for a representative atom, assuming that the neighboring atoms follow a uniform deformation.
Later, QC was extended to multilattices \cite{TadmorSmithBernsteinEtAl1999} (a multilattice is a union of a number of Bravais lattices) based on the improved Cauchy--Born rule \cite{Stakgold1950} which accounts for relative shifts between the Bravais lattices.
Examples of such materials include diamond cubic \mbox{Si}, HCP metals (stacking two simple hexagonal lattices with a shift vector) like \mbox{Zr}, ferroelectric materials, salts like sodium chloride, and intermetallics like \mbox{NiAl}.
More recent developments of QC for multilattices also include adaptive choice of representative cell of multilattices \cite{DobsonElliottLuskinEtAl2007}.
It appears that no rigorous analysis is available so far for the multilattice QC except for the authors' preprint \cite{AbdulleLinShapeev2010}.

In the present work we propose a treatment of multilattices within the framework of numerical homogenization.
Homogenization techniques for partial differential equations (PDEs) with multiscale coefficients
are known to be successful for obtaining effective equations with coefficients properly averaged out
\cite{BensoussanLionsPapanicolaou1978}.
Finite element methods based on homogenization theory have been pioneered by
Bab\u{u}ska \cite{Babuska1976} and have attracted growing attention in recent years
(see \cite{Abdulle2009, Abd11b, EEL2007, EfendievHou2009, GeersKouznetsovaBrekelmans2010} for textbooks or review papers).
Following the ideas of \cite{BensoussanLionsPapanicolaou1978}, we use formal homogenization techniques to describe
the coarse-graining of multilattices and based on that, propose a macro-to-micro numerical algorithm which we call the {\it homogenized QC (HQC)} method.
Here the term {\it macro-to-micro} refers to coupling macroscopic and microscopic scales for the
same physical model, but not coupling models, like in the nonlocal QC.
The macro-to-micro method developed in this paper follows the framework of the finite element 
heterogenenous multiscale method (FE-HMM) \cite{Abdulle2009, Abd11b, EEL2007}, a numerical
method coupling a macroscopic finite element method (FEM) defined on a macroscopic mesh
with effective data recovered on the fly by microscopic FEM on patches centered at suitable quadrature points
within the macroscopic mesh. This method belongs to the family of numerical homogenization methods
as it provides a homogenized numerical solution, but unlike classical methods, the effective data are not precomputed
but supplemented by micro computations when and where needed during the macro computation. The HMM 
provides an efficient way of coupling micro and macro solvers and a suitable framework for
a priori and a posteriori analysis taking into account numerical approximation at different scales \cite{Abdulle2009, Abd11b}.

We give a unified mathematical description and establish equivalence between the homogenized QC, the multilattice QC (MQC) of \cite{TadmorSmithBernsteinEtAl1999}, and the finite element method applied to the continuously homogenized equations (see \cite{AbdulleLinShapeev2010, FishChenLi2007} and references therein for homogenization of atomistic media).
Despite the formal equivalence, we find value in formulating MQC within the homogenization framework and, more generally, in connecting the existing developments in upscaling atomistic models and classical numerical homogenization.
First, this framework allows us to apply the numerical analysis techniques developed for continuum numerical homogenization such as the finite element heterogeneous multiscale method \cite{Abdulle2009, EEL2007} to the multilattice QC method (see our preprint \cite{AbdulleLinShapeev2010} for an example of such application).
Second, numerical homogenization techniques can be used to upscale the atomistic model in both, time and space, which makes it promising for modeling and especially analyzing motion of atomistic materials at macro- and micro-scale \cite{EEL2007, FishChenLi2007, MillerTadmor2002}.
In this work we demonstrate such an application of HQC to a slow (i.e., with no thermal fluctuation) dynamics of an atomistic crystal (Section \ref{sec:unsteady}).
Also, numerical techniques based on the  homogenization framework are well suited for materials described on stochastic lattices 
at the atomistic level 
such as  polymers \cite{BaumanOdenPrudhomme2009} and glasses (see, e.g., \cite{AlicandroCicaleseGloria2011, BlancLeLions2007_stochastic_lattices}),
or for materials with properties (such as, e.g., conductivity, stiffness, etc.$\mathstrut$)
described by random parameters at the continuum level \cite{Tor005}.
We give an example of application of numerical homogenization to a stochastic material in Section \ref{sec:stochastic}.
We note that the idea of applying numerical homogenization methods to atomistic media has appeared in the literature before \cite{BaumanOdenPrudhomme2009, ChenFish2006, Chung2004, ChungNamburu2003, FishChenLi2007}.
\\

The paper is organized as follows.
We present the atomistic model in Section \ref{sec:problem_formulation}, and in particular we give a simplified illustrative model in Section \ref{sec:problem_formulation:simplified}.
The simplified model will be useful to better illustrate application of the coarse-graining methods and to draw analogies between the concepts discussed in this paper and their counterparts in the classical continuum homogenization.
We then present the quasicontinuum method in Section \ref{sec:QC}.
In Section \ref{sec:homogenization} we present a formal homogenization technique applied to the atomistic equations.
In Section \ref{sec:HQC} we present the HQC method---a concurrent macro-to-micro algorithm based on the discrete homogenization.
Section \ref{sec:equivalence} is devoted to showing the equivalence of the following three methods applied to multilattices: the HQC method, the MQC method, and the finite element method applied to continuously homogenized equations.
In Section \ref{sec:multilattice} we illustrate an application of HQC to a multilattice.
We emphasize that the HQC is formulated in such a way that it allows for a straightforward extension to non-crystalline materials if the microstructure is known; an example of such extension is given in Section \ref{sec:stochastic}.
In Section \ref{sec:unsteady} we apply the proposed macro-to-micro method to a long-wave unsteady evolution of a 1D multilattice crystal.
Concluding remarks are given in Section \ref{sec:conclusion}.
The commonly used notations are collected in the appendix.

\section{Problem Formulation}\label{sec:problem_formulation}
The focus of the present study is on correct treatment of atomistic materials with spatially oscillating or inhomogeneous local properties.

\subsection{Equations of Equilibrium}\label{sec:problem_formulation:full}

We describe the formulation of the problem of finding an equilibrium of an atomistic material in the periodic setting.
We consider the periodic boundary conditions for simplicity, in order to avoid difficulties arising from presence of the boundary of the atomistic material.
Nevertheless, it should be noted that the numerical method and the algorithm proposed in the present work can be applied to Dirichlet, Neumann, or other boundary conditions.

\subsubsection{Deformation}
Consider an atomistic material occupying a region $\Omega=[0,1)^d$ in its reference (i.e., undeformed) configuration and extended periodically outside of $\Omega$.
The set of positions of atoms in the reference configuration is
\[
\calM = \Omega\cap\bigcup_{\alpha=0}^{m-1} \big(\eps\bbZ^d+\eps p_\alpha\big),
\]
where $p_\alpha\in[0,1)^d$ is a shift vector of $\alpha$-th species of atoms in the reference configuration; in total we have $m$ species of atoms.
We assume that $p_\alpha\ne p_\beta$ for $\alpha\ne\beta$ and, for convenience, $p_0=0$.

We collect these shift vectors into the set $\calP := \{p_\alpha \,:\ \alpha=0,\ldots,m-1\}$.
Thus, if we denote a Bravais lattice in $\Omega$ by
\[
\calL = \Omega\cap\eps\bbZ^d,
\]
then we can write $\calM = \calL + \eps\calP$.
This identity means that $\calM$ consists of $\eps\calP$ repeated periodically with the period $\eps$.
We will call $\calM$ a multilattice.
The sets $\calL$, $\calP$, and $\calM$ are illustrated in Figure \ref{fig:LPM}.

When the material experiences a deformation, the atom positions become $x + u(x)$, where $u(x)$ is the displacement.
We assume that $u(x)$ is periodic; i.e., $u(x+a)=u(x)$ for all $a\in\bbZ^d$.
The space of all periodic displacements is denoted by $\calU_\per(\calM)$.
Since we consider only the systems invariant with respect to translation in space, we will also need the space of displacements with zero average, $\calU_\#(\calM)$ (see Appendix \ref{sec:notations:spaces} for the precise definitions).

\begin{figure}
\begin{center}
\includegraphics{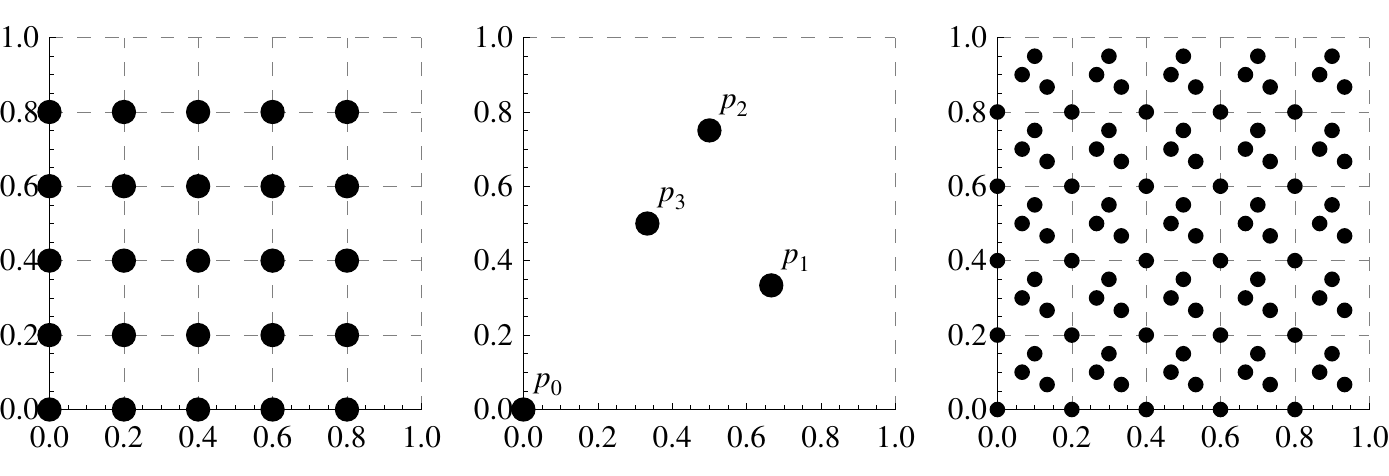}
\caption{Illustration of $\calL$ (left), $\calP=\{p_0, p_1, p_2, p_3\}$ (middle), and $\calM$ (right); here $\eps=1/5$.}
\label{fig:LPM}
\end{center}
\end{figure}

\subsubsection{Interaction}
We assume a general (multibody) finite-range interaction between atoms.
For each atom $x\in\calM$ we introduce its ``interaction neighborhood''---a set of vectors $\calR_\eps(x)$ such that $\{x+\eps r\,:\,r\in\calR_\eps(x)\}$ are the atoms that $x$ interacts with.
The energy of an atom $x\in\calM$ is denoted by $V_\eps(D_{\calR_\eps(x)} u(x); x)$, where $D_{\calR_\eps} u = (D_r u)_{r\in\calR_\eps}$ (see Appendix \ref{sec:notations:vector_indexed}) is a collection of discrete directional derivatives of $u$ corresponding to the set of neighbors $\calR_\eps$ (these notations were first introduced in \cite{HudsonOrtner2012}).
The discrete derivative in direction $r$ of $u$ evaluated at $x\in\calM$ is defined as $D_r u(x) := \frac{u(x+\eps r)-u(x)}\eps$.
The needed properties and definitions of discrete directional derivatives can be found in Appendix \ref{sec:notations:operators}, and more details on discrete directional derivatives in Appendix \ref{sec:notations:vector_indexed}.

Thus, the interaction energy of the displacement $u$ is given by the interaction potential $V_\eps$ as
\[
E(u)
= \frac1{\#(\calM)}\sum_{x\in\calM} V_\eps(D_{\calR_\eps(x)} u(x); x)
= \big\< V_\eps(D_{\calR_\eps} u) \big\>_{\calM}
,
\]
where $\<g\>_S$ denotes the average value of a function $g$ defined on a discrete set $S$.

The subscript $\eps$ in $V_\eps$ and $\calR_\eps$ indicates that these objects depend nonsmoothly on $x$: indeed, the interaction energy and the interaction neighborhood may depend on the species of atoms $\alpha$ for $x\in\calL+\eps p_\alpha$.
For instance, we can consider a Lennard--Jones potential with atom-dependent parameters:
\begin{equation}
\label{eq:problem_formulation:LJ}
V_\eps(D_{\calR_\eps} u; x) = \sum_{r\in\calR_\eps} s_{x,x+\eps r}\Big(-2\, \big(\smfrac{|r+D_r u|}{\ell_{x,x+\eps r}}\big)^{-6}+\big(\smfrac{|r+D_r u|}{\ell_{x,x+\eps r}}\big)^{-12}\Big),
\end{equation}
where $s_{x,x+\eps r}$ and $\ell_{x,x+\eps r}$ are, respectively, the strength and the equilibrium distance of interaction of atoms $x$ and $x+\eps r$.

We assume that the interaction neighborhood $\calR_\eps(x+\eps p_\alpha)$ and the interaction potential $V_\eps(\bullet,x+\eps p_\alpha)$ for $x\in\calL$ depend only on $\alpha$, the particular species of atoms, but do not depend on $x$; we therefore write $\calR_\eps(x+\eps p_\alpha)=:\calR_{\eps,\alpha}$ and $V_\eps(\bullet,x+\eps p_\alpha) =: V_{\eps,\alpha}$.
This assumption states, effectively, an $\eps$-periodicity of $\calR_\eps$ and $V_\eps(\bullet,x)$.
Then, we can use the following form of the energy:
\begin{align} \notag
E(u)
=~&
\bigg\<
	\frac1m \sum_{\alpha=0}^{m-1}
	V_\eps(D_{\calR_\eps(x+\eps p_\alpha)} u(x+\eps p_\alpha); x+\eps p_\alpha)
\bigg\>_{x\in\calL}
\\ =~& \label{eq:E_alt}
\bigg\<
	\frac1m \sum_{\alpha=0}^{m-1}
	V_{\eps,\alpha}(D_{\calR_{\eps,\alpha}} u(x+\eps p_\alpha))
\bigg\>_{x\in\calL},
\end{align}
where we used a more verbose notation for averaging of a function $g$ defined on a discrete set $S$, $\<g\>_S=:\<g(x)\>_{x\in S}$.
This expression for the energy will be used to write down the energy of the MQC method in a familiar way (see \eqref{eq:Emqc-simplified}).

\begin{remark}
One can exercise the freedom in choosing $\calP$ by assuming that $\calP = \big\{0,\smfrac1m e_1, \ldots,\smfrac{m-1}m e_1\big\}$, where $e_1\in\bbR^d$ is the respective unit vector.
In this case $\calM$, up to a dilatation, is a simple lattice (although with several species of atoms).
This allows one to choose $\calR_\eps(x)$ independent of $x$ (and also $\calR_{\eps,\alpha}$ independent of $\alpha$), and leave only interaction potential $V_\eps$ to depend on $x$.

We will not pursue this in the present work; however, such notations would significantly simplify presentation of the MQC (Section \ref{sec:MQC}) and would allow one to conveniently write the equilibrium equation in a strong form (in particular, in Section \ref{sec:homogenization:fast-and-slow}).
Our motivation for not pursuing this is to show that the homogenization and the numerical method can, in principle, be generalized to the case when $\calR_\eps$ depends on $x$.
This is important when modeling non-crystalline materials with no underlying periodic structure.
\end{remark}

\subsubsection{External Force}
The potential energy of the external force $f=f(x)$ is
\[
-F(u) = - \<f, u \>_\calM,
\]
where by $\< w, v \>_\calM := \< w \cdot v \>_\calM$ we denote a scalar product of $w,v\in\calU_\per(\calM)$.
(To be precise, it is an inner product on $\calU_\#(\calM)$ and a semi-inner product on $\calU_\per(\calM)$.)
The forces $f=f(x)$ are applied as ``dead loads''; i.e., they are independent of actual atom positions $x+u$.
For the problem to be well-posed, the sum of all forces per period is assumed to be zero; i.e., $\<f\>_\calM = 0$.

\subsubsection{Equation of Equilibrium}
We denote the total potential energy of the atomistic system by
\[
\Pi(u) = E(u) -F(u).
\]
A displacement $u\in\calU_\#(\calM)$ is a stable equilibrium if it is a local minimizer of $\Pi$, which implies that $u$ is a critical point of $\Pi$:
\begin{equation}
\label{eq:original_equation}
\<\delPi(u), v\>_\calM := \frac{\dd}{\dd t} \Pi(u+t v)\big|_{t=0} = 0
\quad \forall v\in\calU_\#(\calM).
\end{equation}
We assume that the function $\Pi(u)$ is smooth enough, and hence $\<\delPi(u), v\>_\calM$ is a linear functional with respect to $v\in\calU_\#(\calM)$, which justifies identification of $\delPi(u)$ with an element of $\calU_\#$.
Alternatively, the problem of finding the equilibrium configuration of atoms can formally be written as
\[
\frac{\partial \Pi}{\partial u(x)} = 0 \quad \forall x\in\calM,
\]
if we consider $\Pi$ as a function of finite number of variables $u(x)$, $x\in\calM$.

A physical potential energy $\Pi(u)$ has to be invariant with respect to a uniform translation of atoms.
Hence, we pose the following additional condition,
\begin{equation}
\<u\>_\calM = 0,
\label{eq:u-averages-to-zero}
\end{equation}
which is necessary (but may not be sufficient) for the equations \eqref{eq:original_equation} to have a \change{locally unique} solution.

The equilibrium equations \eqref{eq:original_equation} together with the additional condition \eqref{eq:u-averages-to-zero} can be written in variational form: find $u \in \calU_{\per}(\calM)$ such that
\begin{subeqnarray} \label{eq:variational_equation}
	\<\delE(u), v\>_\calM & = & F(v)
	\quad \forall v\in \calU_{\per}(\calM)
\\
	\<u\>_\calM & = & 0,
\label{eq:variational_problem_generic}
\end{subeqnarray}
where the functional derivative $\delE: \calU_{\per}(\calM) \to \calU_{\per}(\calM)$ is computed as
\begin{align}
\label{eq:Psi}
\<\delE(u), v\>_\calM
=~&
\Big\<\sum_{r\in\calR_\eps} V'_{\eps,r}(D_{\calR_\eps} u), D_r v\Big\>_\calM
,
\end{align}
and $V'_{\eps,r}(D_{\calR_\eps} u)$ denotes, effectively, the gradient of a scalar function $V_\eps$ with respect to its vector-valued variable $D_r u$ (note the difference with $V_{\eps, \beta}$ introduced in \eqref{eq:E_alt}).
Here and in what follows, with a slight abuse of notations, we keep the sign of summation over $r\in\calR_\eps$ inside the triangular brackets of the scalar product.

\subsection{A Simple Illustrative Example}\label{sec:problem_formulation:simplified}

The following simplified model will be useful in illustrating the concepts presented in this paper (namely, we will give a simplified version of the quasicontinuum method, in Section \ref{sec:QC-heterogeneous:failure}, and illustrate an application of the homogenization, in Section \ref{sec:homogenization:simplfied}).
The reader can find more examples involving a simplified model in our preprint \cite{AbdulleLinShapeev2010}.

Assume one space dimension, $d=1$; the domain $\Omega=[0,1)$, the shift vectors in the reference configuration
\begin{equation} \label{eq:simplified_model:P}
\calP = \big\{0,\smfrac1m,\ldots,\smfrac{(m-1)}m\big\},
\end{equation}
the multilattice
\[
\calM
=\bigcup_{\alpha=0}^{m-1}(\eps \bbZ + \eps \smfrac{\alpha}{m}) \cap \Omega
=\smfrac\eps m\bbZ\cap\Omega
,
\]
and the basic lattice $\calL=\eps\bbZ\cap\Omega$.
We further assume $\calR=\{\smfrac1m\}$ (nearest neighbor interaction only) and consider the ``linear spring model'' with the atomistic potential
\begin{equation} \label{eq:simplified_model}
V_\eps(D_r u; x) = \psi_\eps(x)\, \frac{(D_r u)^2}2
,
\end{equation}
with $r=\smfrac1m$.
Such a system can be interpreted as a system of masses located at positions $x + u$ and connected with ideal springs with spring constants $k_\alpha = \psi_\eps(x)/\eps$ (where $\alpha$ and $x$ are related here through $x\in \eps \smfrac{1+\alpha}{m} + \eps\bbZ$), as illustrated in Figure \ref{fig:springs_heterogeneous}.

\begin{figure}
\begin{center}
\includegraphics{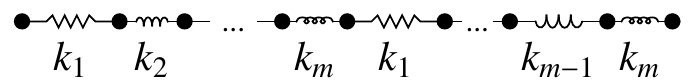}
\caption{Illustration of a simplified atomistic model}
\label{fig:springs_heterogeneous}
\end{center}
\end{figure}

The equilibrium equation then becomes
\begin{equation}
	\<\psi_\eps D_r u, D_r v\>_\calM
	= \<f, v\>_\calM
.
\label{eq:problem_1d_linear}
\end{equation}

If we want to find an equilibrium of a very large atomistic system, we need to coarse-grain these equations.
In Section \ref{sec:QC} we present the quasicontinuum method, one of the methods of numerical coarse-graining of such a system.

We can notice that the equation \eqref{eq:problem_1d_linear} closely resembles the continuum equation
\begin{equation}\label{eq:intro:continuum_energy}
\int_\Omega A\big(\smfrac x\eps\big) \frac{\dd u}{\dd x} \frac{\dd v}{\dd x} \dd x
=
\int_\Omega f v \dd x
,
\end{equation}
for which the homogenization theory is well-developed.
Here $A\big(\smfrac x\eps\big)$ is an oscillating coefficient defining the local energy density.
The crystal is, by definition, a periodic arrangement of atoms, which translates into periodicity of $A\big(\smfrac x\eps\big)$.
Non-crystalline solid materials, in contrast, correspond to random arrangements of atoms, which is analogous to random (non-periodic) $A\big(\smfrac x\eps\big)$.
The spring constants varying on the scale of $\eps$ are analogous to $A\big(\smfrac x\eps\big)$ varying on the scale of $\eps$.
It is well known from homogenization theory (\cite{BensoussanLionsPapanicolaou1978}) that the solution $u$ of \eqref{eq:intro:continuum_energy} converges weakly
in the $H^1$ norm to a homogenized solution $\bar u,$ solution to an equation similar to \eqref{eq:intro:continuum_energy} but
with an effective (homogenized) tensor $\bar A(x).$ We note that, in general, strong convergence holds only for the $L^2$ norm.

Based on this similarity between the continuum and the discrete energy, we apply the formal homogenization techniques to the discrete atomistic equations in Section \ref{sec:homogenization} and based on that formulate the HQC method---a concurrent macro-to-micro algorithm (similar to FE-HMM) based on the discrete homogenization.
The method is formulated in such a way that it allows for a straightforward extension to non-crystalline materials if the microstructure is known; an example of such extension is given in Section \ref{sec:stochastic}. 

In Section \ref{sec:unsteady} we apply the proposed macro-to-micro method to a long-wave unsteady evolution of a 1D multilattice crystal.
The long-wave unsteady evolution is analogous to a continuum motion corresponding to a Hamiltonian
\[
\smfrac12 \int_\Omega \bigg[m\big(\smfrac x\eps\big) \Big(\frac{\dd u}{\dd t}\Big)^2
+  A\big(\smfrac x\eps\big) \Big(\frac{\dd u}{\dd x}\Big)^2
\bigg]\dd x,
\]
where $u=u(t,x)$ is assumed to have no fast (i.e., on the time scale of $\smfrac1\eps$) oscillations.

\section{Quasicontinuum (QC) Method}\label{sec:QC}

Traditionally, numerical methods such as the finite element method (FEM) are applied to continuum equations which can then be solved on a computer.
The characteristic feature of the atomistic models we are discussing in the paper is their discreteness, with a number of degrees of freedom often too large to keep track of each individual atom.
Therefore, similarly to FEM, the ideas of reducing the number of degrees of freedom are used for atomistic models as well. The difference is that now the reduction is done from a large but finite number of degrees of freedom to a smaller number of degrees of freedom.
The QC method is a representative of such methods.
We first present its simple-lattice version.
The QC method consists of reducing the number of degrees of freedom of the atomistic system by choosing a coarse mesh of nodal atoms and assuming that the positions of the other atoms can be reconstructed by a linear interpolation.

It should be noted that we discuss here only the {\it local} version of QC which is equivalent to applying FEM to the Cauchy--Born continuum model of elasticity.
We are not considering coupling the continuum and discrete models in this paper.

\subsection{Notation}\label{sec:QC:notations}

Assume a partition $\calTh$ of the domain $\Omega$ into simplicial elements $T$, which we will conveniently refer to as the {\it mesh}.
Normally, $\#(\calTh)\ll\#(\calL)$ (recall that by $\#(\bullet)$ we denote the number of elements in a set).
By $|T|$ we denote the Lebesgue measure of $T$.
The QC solution will be denoted by $u^h$.

The space of piecewise linear discrete vector-functions is denoted by
\begin{equation}
\calU^h_\per = \big\{
	u^h\in \big(W^{1,\infty}_\per(\Omega)\big)^d\,:\ u^h|_T\in P_1(T) ~\forall T\in\calTh\big\}
,
\label{eq:UH-space}
\end{equation}
and the space of piecewise constant vector-functions as
\begin{displaymath}
\calQ^h_\per = \big\{
	q^h\in \big(L^{\infty}_\per(\Omega)\big)^d\,:\ q^h|_T\in P_0(T) ~\forall T\in\calTh\big\}.
\end{displaymath}

\subsection{QC for simple lattice}\label{sec:QC:homogeneous}

In this (and only this) subsection we make the simple lattice assumption.
That is, we assume that $m=1$ and hence $\calM=\calL$.
In particular, in this subsection we write $V_\eps(D_\calR u; x)=V(D_\calR u)$ and $\calR_\eps(x)=\calR$ as they no longer depend on $x$.

The QC method \cite{TadmorPhillipsOrtiz1996} aims at finding a minimizer of
\[
\Pi(u^h)
=
\big\<V\big(D_\calR u^h\big)\big\>_\calL
- F(u^h)
\]
in $\calU^h_\per$.
Minimizing $\Pi(u^h)$ in $\calU^h_\per$ indeed reduces the number of degrees of freedom of the system from $O(\#(\calM))$ to $O(\#(\calTh))$ (recall that $\#(\calTh)\ll \#(\calM)$).
However, one must still spend $O(\#(\calM))$ operations to compute the effective forces on the reduced degrees of freedom.
In order to have an efficient numerical method (i.e., a method with $O(\#(\calT_h))$ operations) one introduces an approximation to $\Pi(u^h)$ which is called the  {\it local QC method} \cite{TadmorPhillipsOrtiz1996} (hereinafter referred to as the QC method).

The local QC method first approximates $D_r u^h$ with $\nabla_r u^h$ within each $T$ (hence the name of the method: the nonlocal finite difference $D_r u^h$ is approximated with the ``local'' directional derivative $\nabla_r u^h$).
Then for each $x\in T$ one has
\begin{align*}
V(D_\calR u^h)
\approx~&
V(\nabla_\calR u^h)
= W\big(\nabla u^h|_T\big),
\end{align*}
where $W(\mF) := V(\mF\calR)$ is the Cauchy--Born energy density associated with a displacement gradient $\mF$ (see \eqref{eq:vector_indexed} to obtain the precise definition of $\mF\calR$).
Second, the local QC method changes the sum over $x\in\calL$ effectively to integration over $\Omega$; i.e.,
\[
E^\qc(u^h) :=
\int_\Omega W\big(\nabla u^h\big) \dx
=
\sum_{T\in\calTh} |T|\, W\big(\nabla u^h|T\big)
.
\]

The variational formulation of the QC method is thus
\begin{equation}
\label{eq:qc}
\int_\Omega
\sum_{r\in\calR} \del W\big(\nabla_r u^h\big) \!:\! \nabla_r v^h
\dx
= F^h(v^h)
\quad \forall v^h\in\calU^h_\per
,
\end{equation}
where $\del W$ denotes the derivative of $W$, the semicolon denotes the inner product of matrices.
and $F^h(v^h)$ is some approximation to $\<f,v^h\>_\calM$.

Error analysis of the local QC yields a first-order convergence of the deformation gradient (i.e., roughly speaking, of a quantity $\|u^h-u\|_{W^{1,p}(\Omega)}$) with respect to sizes of triangles $T\in\calT_h$ (see, e.g., \cite{Lin2003, Lin2007, OrtnerShapeev2011}).
A more refined analysis shows that the local QC can be second-order accurate \cite{DobsonLuskinOrtner2010b, EMing2007, MakridakisSuli}.

\subsection{Multilattice QC}\label{sec:MQC}

Approximating the exact minimizer of $\Pi(u)$ with a piecewise linear $u^h \in \calU^h_\per$ may be accurate enough for the case when the interatomic interaction $V_\eps(\bullet, x)$ varies smoothly with $x$ (more precisely, if the mesh $\calTh$ resolves the variations in $V_\eps(\bullet, x)$) well.
However, for many materials with multilattice structure (examples of such materials were given in the introduction) the piecewise linear approximation of the displacement $u$ is not accurate.

In this subsection we present the Multilattice QC (MQC) method first introduced in \cite{TadmorSmithBernsteinEtAl1999} which is designed to handle the multilattice microstructure.

Define the space of QC displacements of the multilattice $\calM$:
\begin{equation}
\calU^{h,q} = \bigg\{
	u^h + \sum_{\alpha=1}^{m-1} q^h_\alpha w_\alpha
	\,:~
	u^h\in \calU^h_\per
	,~q^h_\alpha\in \calQ^h_\per
	,~\alpha=1,\ldots,m-1
\bigg\},
\label{eq:QC_general_complex-lattice_space}
\end{equation}
where $q^h_\alpha$ are the deformed shift vectors (recall that $p_\alpha$ are the undeformed shift vectors) and $w_\alpha : \calM \to \bbR$ are the associated basis functions defined as 
\begin{equation}\label{eq:MQC_basis_funcs_def}
w_\alpha|_{\calL+\eps p_\beta} = \delta_{\alpha\beta}
\quad (\alpha,\beta=0,\ldots,m-1)
,
\end{equation}
with $\delta_{\alpha\beta}$ denoting the Kronecker delta.
It should be noted that the domain of definition of functions in $\calU^{h,q}$ is $\calM$, whereas the functions in $\calU^h$ are defined on the entire $\bbR^d$.
For a more detailed introduction of the space of QC deformations, refer to \cite{AbdulleLinShapeev2010}.
In each element $T\in\calTh$ we thus have $m-1$ nonzero shift vectors $q^h_\alpha$, and we set $q^h_0:=0$.
We denote
\[
\bq^h := (q^h_1,\ldots,q^h_{m-1}) \in (\calQ^h_\per)^{m-1}.
\]

Next, form the interaction energy $E(u)$ with $u = u^h + \sum_{\alpha=1}^{m-1} q^h_\alpha w_\alpha \in\calU^{h,q}$:
\begin{align*}
E(u)
=~&
E\Big(u^h + \sum_{\alpha=1}^{m-1} q^h_\alpha w_\alpha\Big)
\\ = ~&
\Big\<
	V_\eps\Big(
		D_{\calR_\eps(x)} \Big({ u^h(x) + \sum_{\alpha=1}^{m-1} q^h_\alpha(x) w_\alpha(x)}\Big)
		; x
	\Big)
\Big>_{x\in\calM}
\\ = ~&
\Big\<
	\frac1m \sum_{\beta=0}^{m-1}
	V_\eps\Big(
		D_{\calR_\eps(x+\eps p_\beta)} \Big({ u^h(x+\eps p_\beta) + \sum_{\alpha=1}^{m-1} q^h_\alpha(x+\eps p_\beta) w_\alpha(x+\eps p_\beta)}\Big)
		; x+\eps p_\beta
	\Big)
\Big>_{x\in\calL}
\\ = ~&
\Big\<
	\frac1m \sum_{\beta=0}^{m-1}
	V_{\eps,\beta}\Big(
		D_{\calR_{\eps,\beta}} { u^h(x+\eps p_\beta) + \sum_{\alpha=1}^{m-1} D_{\calR_{\eps,\beta}} q^h_\alpha(x+\eps p_\beta) w_\alpha(\eps p_\beta)}
\Big)
\Big>_{x\in\calL}
,
\end{align*}
where we used periodicity of $V_\eps$ (see \eqref{eq:E_alt}) and $w_\alpha$ (which follows directly from the definitions of $w_\alpha$ and $\calM$).
Similarly to the simple-lattice QC, we perform a local quasicontinuum approximation which consists of:
(i) changing the summation over $x\in\calL$ to the integration over $\Omega$,
(ii) approximating $D_r u^h$ with $\nabla_r u^h$, and
(iii) approximating $q^h_\alpha(x+\eps p_\beta)$ with $q^h_\alpha(x)$:
\begin{align*}
E(u)
\approx~&
\int_\Omega
	\frac1m \sum_{\beta=0}^{m-1}
	V_{\eps,\beta}\Big(
		\nabla_{\calR_{\eps,\beta}} u^h + \sum_{\alpha=1}^{m-1} q^h_\alpha D_{\calR_{\eps,\beta}} w_\alpha(\eps p_\beta)
\Big)
\dx
\\ =~&
\sum_{T\in\calTh} |T|\,
	\frac1m \sum_{\beta=0}^{m-1}
	V_{\eps,\beta}\Big(
		\big(\nabla u^h|_T\big)\calR_{\eps,\beta} + \sum_{\alpha=1}^{m-1} \big(q^h_\alpha|_T\big) D_{\calR_{\eps,\beta}} w_\alpha(\eps p_\beta)
\Big)
\\=:~& \tildeE^\mqc(u^h, \bq^h)
,
\end{align*}
where we used the identity $\nabla_{\calR_{\eps,\beta}} u^h|_T = \big(\nabla u^h|_T\big)\calR_{\eps,\beta}$; cf.\ \eqref{eq:vector_indexed}.
\begin{remark}
The expression for $\tildeE^\mqc(u^h, \{q^h_\alpha\})$ can be further simplified by denoting the species of atoms $\eps\beta + \calR_\eps$ as $\calA_{\eps,\beta}$ (formally $\calA_{\eps,\beta} := (a_{\beta, r})_{r\in\calR_{\eps,\beta}}$ where $a_{\beta, r}\in\{0,\ldots,m-1\}$ is defined so that $p_{a_{\beta, r}} \in p_\beta + r + \bbZ^d$).
Then the sum in $\tildeE^\mqc$ can be simplified as the difference between the shift vectors of interacting atoms:
\begin{align*}
\sum_{\alpha=1}^{m-1} \big(q^h_\alpha|_T\big) D_{\calR_{\eps,\beta}} w_\alpha(\eps p_\beta)
=~&
\bigg(
\sum_{\alpha=1}^{m-1} \big(q^h_\alpha|_T\big) D_r w_\alpha(\eps p_\beta)
\bigg)_{r\in\calR_{\eps,\beta}}
\\=~&
\Big(
\big(q^h_{a_{\beta, r}}|_T\big) \big(w_\alpha(\eps p_{a_{\beta, r}})-w_\alpha(\eps p_\beta)\big)
\Big)_{r\in\calR_{\eps,\beta}}
\\=~&
\Big(
\big(q^h_{a_{\beta, r}}|_T\big) - \big(q^h_{\beta}|_T\big)
\Big)_{r\in\calR_{\eps,\beta}}
\end{align*}
This yields
\begin{equation}\label{eq:Emqc-simplified}
\tildeE^\mqc(u^h, \bq^h)
=
\sum_{T\in\calTh} |T|\,
	\frac1m \sum_{\beta=0}^{m-1}
	V_{\eps,\beta}\Big(
		\big(\nabla_{\calR_{\eps,\beta}} u^h + q^h_{\calA_{\eps,\beta}} - q^h_\beta\big)\big|_T
	\Big)
.
\end{equation}

\end{remark}

In the next step, the shift vectors $q_\alpha$ are eliminated from \eqref{eq:QC_general_complex-lattice_space} by requiring that the variation of $\tildeE^\mqc(u^h, \{q_\alpha\})$ with respect to $q_\gamma$ in each triangle be zero:
\begin{equation}
\label{eq:QC_general_shift-vectors-equation}
\begin{split}
	\frac1m \sum_{\beta=0}^{m-1}
	\sum_{r\in\calR_\beta}
	V'_{\eps,\beta,r}\Big(
		\big(\nabla u^h|_T\big)\calR_{\eps,\beta} + \sum_{\alpha=1}^{m-1} \big(q^h_\alpha|_T\big) D_{\calR_{\eps,\beta}} w_\alpha(\eps p_\beta)
\Big)
	D_r w_\gamma(\eps p_\beta) = 0
& \\
\quad (\gamma=1,2,\ldots,m-1)
& .
\end{split}
\end{equation}
The equations \eqref{eq:QC_general_shift-vectors-equation} form a system of $m-1$ equations for $m-1$ unknowns $(q_\alpha)_{\alpha=1}^{m-1}$ in each $T$.
A solution of this system gives us the shift vectors $q_\alpha$ depending (as a rule, nonlinearly) only on the displacement gradient:
\[
\bq^h|_T = \bq\big(\nabla u^h|_T\big).
\]
Note that the function $\bq(\mF)$ does not depend on $T$, unless different periodic materials are considered in different elements $T$.

\begin{remark}\label{rem:no_discussion_of_existence}
The function $\bq(\mF)$ determines the lattice microstructure of a material under the macroscopic displacement gradient $\mF$.
Often there is more than one lattice microstructure corresponding to a particular $\mF$.
Well-posedness of equations \eqref{eq:QC_general_shift-vectors-equation} is studied in \cite{EMing2007} under the assumption that the entire atomistic system is $H^1$-stable, and in \cite{AbdulleLinShapeev2011_analysis} under the assumption of dominance of nearest-neighbor interaction in 1D.

In different applications there may be different \change{additional} conditions for choosing the unique $\bq(\mF)$ (this can be the condition of a global minimum of the microenergy, or proximity to a given microfunction).
\change{In this paper we will not focus on such additional conditions, and}
will \change{therefore} not discuss in detail the existence and uniqueness of solutions of the respective microscopic and macroscopic equations.
\change{Thus}, at this point, by $\bq(\mF)$ we formally denote one of the solutions of \eqref{eq:QC_general_shift-vectors-equation}, or leave $\bq(\mF)$ undefined if \eqref{eq:QC_general_shift-vectors-equation} admits no solutions.
In Section \ref{sec:equivalence} we will take a slightly more formal account of existence and uniqueness.
\end{remark}

We now form a QC energy with $q_\alpha$ eliminated:
\begin{equation} \label{eq:Emqc}
E^\mqc(u^h)
:=
\tildeE^\mqc\big(u^h, \bq\big(\nabla u^h\big)\big)
.
\end{equation}
The QC equation of equilibrium now reads: find $u^h\in\calU^h_\per$ such that
\[
\<\delE^\mqc(u^h), v^h\>_\Omega
=
F^h(v^h)
\quad\forall v^h\in\calU^h_\per,
\]
where $F^h(v^h)$ is some approximation to $\<f,v^h\>_\calM$.
The function $u^h$ gives a macroscopic displacement of the material, and one needs to compute $u^h + \sum_{\alpha=1}^{m-1} q^h_\alpha w_\alpha$ for the microstructure.
We note that since $q_\alpha$ were found by letting the variation of $\tildeE^\mqc(u^h, q_\alpha)$ with respect to $q_\alpha$ be zero, we have
\begin{equation}
\label{eq:Eintqc_simplified}
\delE^\mqc(u^h) = \deltildeE^\mqc\big(u^h, \bq\big(\nabla u^h\big)\big).
\end{equation}

\begin{remark}\label{rem:concurrent_coupling}
Instead of eliminating $\bq^h=\bq(\nabla u^h)$, one could also look for a critical point (or a minimizer) of the energy $\tildeE^\mqc(u^h, \bq^h)$ with respect to both $u^h$ and $\bq^h$ \change{(see, e.g., \cite{SorkinElliottTadmor})}.
\end{remark}

\subsection{Application of QC to the Simplified Model} \label{sec:QC-heterogeneous:failure}

We illustrate an application of QC to the simplified 1D model \eqref{eq:simplified_model} for two species of atoms (i.e., $m=2$), $\psi_\eps(0)=\psi_1$, $\psi_\eps\big(\smfrac\eps2\big)=\psi_2$.

If we approximate the exact solution with a piecewise affine displacement $u^h\in\calU^h_\per$ (i.e., without introducing shift vectors, as done in the simple-lattice QC) then we will find the approximate energy
\[
\sum_{T\in\calTh} |T|\,
	\frac12 \bigg[
		\psi_1\, \frac{(\nabla_r u^h)^2}2
		+
		\psi_2\, \frac{(\nabla_r u^h)^2}2
	\bigg]
=
\sum_{T\in\calTh} |T|\,
	\frac{\psi_1+\psi_2}{2}\, \frac{(\nabla_r u^h)^2}2.
\]
Here $\tilde\psi^0 = \frac{\psi_1 + \psi_2}{2}$ is the wrong effective spring constant, since if the two springs in series are replaced with two identical springs with the effective spring constant $\psi_0$ then $\psi_0 = \frac{2\,\psi_1\psi_2}{\psi_1+\psi_2}$ (see, e.g., \cite{ChenFish2006}).

If instead we allow for nonzero shift vector $q_1$ then the corresponding MQC energy \eqref{eq:Emqc-simplified} is
\[
\tildeE^\mqc(u^h, q^h_1)
=
\sum_{T\in\calTh} |T|\,
	\frac12 \bigg[
		\psi_1\, \frac{(\nabla_r u^h + q^h_1)^2}2
		+
		\psi_2\, \frac{(\nabla_r u^h - q^h_1)^2}2
	\bigg]
\]
with $r=\smfrac12$.
The strong form of \eqref{eq:QC_general_shift-vectors-equation} in this case can be obtained by differentiating the above expression with respect to $q^h_1$ in each $T$:
\[
\psi_1 \big((\nabla_r u^h + q^h_1)|_T\big) -
\psi_2 \big((\nabla_r u^h - q^h_1)|_T\big)
=
0,
\]
from where we find
\[
q^h_1|_T = \frac{\psi_2-\psi_1}{\psi_1+\psi_2} \big(\nabla_r u^h|_T\big)
.
\]
Substituting this back into the MQC energy (cf.\ \eqref{eq:Emqc}) yields
\begin{align*}
E^\mqc(u^h)
=~&
\sum_{T\in\calTh} |T|\,
	\frac12 \bigg[
		\psi_1\, \frac12 \Big(\frac{2\,\psi_2}{\psi_1+\psi_2}\,\big(\nabla_r u^h|_T\big)\Big)^2
		+
		\psi_2\, \frac12 \Big(\frac{2\,\psi_1}{\psi_1+\psi_2}\,\big(\nabla_r u^h|_T\big)\Big)^2
	\bigg]
\\ =~&
\sum_{T\in\calTh} |T|\, \frac{2\,\psi_1\psi_2}{\psi_1+\psi_2}\,\frac{(\nabla_r u^h|_T)^2}2
,
\end{align*}
where the effective spring constant $\psi_0 = \frac{2\,\psi_1\psi_2}{\psi_1+\psi_2}$ is now computed correctly.

\section{Homogenization of Atomistic Media}\label{sec:homogenization}

\begin{figure}
\begin{center}
\includegraphics{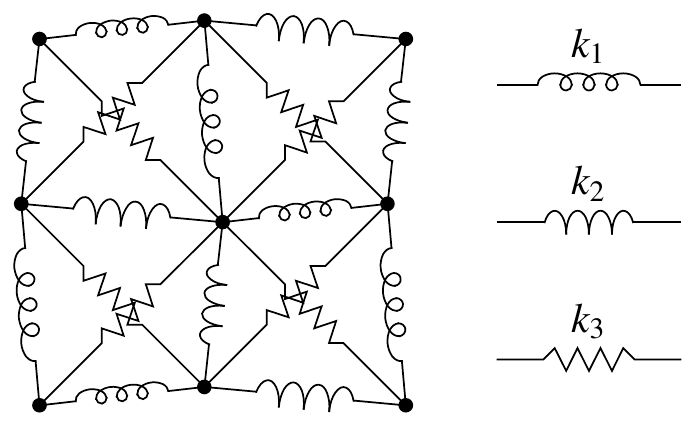}
\caption{Illustration of a 2D model problem with heterogeneous interaction.}
\label{fig:2d-springs}
\end{center}
\end{figure}

We now present another coarse graining strategy based on homogenization.
We derive below the homogenized model of the atomistic material which will be the basis for formulating and analyzing a quasicontinuum method for multilattices.

We note that there are existing works applying formal homogenization techniques to upscaling atomistic equations, see \cite{ChenFish2006, ChenFish2006a, FishChenLi2007} and references therein.
In the present section we derive the upscaled equations for a general model of interaction in many dimensions as opposed to the pairwise interaction in 1D assumed in the upscaled equations \cite{ChenFish2006, ChenFish2006a, FishChenLi2007}.
The upscaled equations are derived using a formal asymptotic expansion.
Rigorous error bounds for the homogenized equations can be found in the preprint \cite{AbdulleLinShapeev2010} for the case of the linear 1D nearest-neighbor interaction and in \cite{AbdulleLinShapeev2011_analysis} for the case of a 1D finite-range nonlinear interaction.

\subsection{Asymptotic expansion}\label{sec:homogenization:fast-and-slow}
In order to take into account the local variation of the atomistic interaction we think of the displacement as depending on a fast and a slow scale $u(x)\sim u(x,x/\eps)$.
We define $x\in\bbR^d$, the macro (``slow'') variable, and $y \in \bbZ^d+\calP$, the micro (``fast'') variable related to $x$ as $y=x/\eps$, and consider a series of functions
$u^n:\bbR^d\times (\bbZ^d+\calP)\rightarrow \bbR^d$ indexed by $n=0,1,2\ldots$
As we consider the local structure and interaction to be periodic, we assume that the functions $u^n$ are $\calP$-periodic in the fast variable; i.e., they satisfy for all $(x,y)\in \Omega\times\calP$
\[
u^n(x,y+j) = u^n(x,y),
\quad\forall j\in\bbZ^d \]
while the behavior with respect to $x$ is similar as considered in the previous sections
\[
u^n(x+i,y) = u^n(x,y),
\quad\forall i\in\bbZ^d .
\]
We then consider the asymptotic expansion
\begin{equation}
\label{equ:asympt_exp}
u (x) \sim \big(u^0(x) + \eps u^1(x, y) + \eps^2 u^2(x, y) + \ldots\big)\big|_{y=x/\eps}
\quad\forall x\in\calM
.
\end{equation}
Notice that we directly assume that the homogenized solution, $u^0$, does not depend on $y$.

We now proceed as in the ``classical homogenization" \cite{Bakhvalov1974, BensoussanLionsPapanicolaou1978, S'anchez-Palencia1980}
and insert the ansatz \eqref{equ:asympt_exp} into \eqref{eq:variational_equation}:
\begin{align*}
\Big\<
\Big(\sum_{r\in\calR_\eps} V'_{\eps,r}\big(
D_{x,\calR_\eps} u^0+\eps D_{x,\calR_\eps} T_{y,\calR_\eps} u^1+ D_{y,\calR_\eps}u^1
+\ldots
\big)&
\\,
D_{x,r} T_{y,r} v + \eps^{-1} D_{y,r} v\Big)&\Big|_{y=x/\eps} \Big\>_\calM
= \<f, v\>_\calM
,
\end{align*}
where the test functions $v=v(x,y)$ are continuous and smooth in $x\in\Omega$ and discrete in $y\in\calP$.
Here we used the relation \eqref{eq:full_and_partial} to expand the full derivative $D_r$ through partial derivatives $D_{x,r}$, $D_{y,r}$, and the translation operator $T_{y,r}$, and used the collection-of-derivatives notation $D_\calR$ (see Appendix \ref{sec:notations:vector_indexed} for more details).

We then extend the equation on the entire $\calM\times\calP$:
\begin{equation}\label{eq:homogenization:extended_discrete}
\begin{split}
\Big\<
\sum_{r\in\calR_\eps} V'_{\eps,r}\big(
D_{x,\calR_\eps} u^0+\eps D_{x,\calR_\eps} T_{y,\calR_\eps} u^1+ D_{y,\calR_\eps}u^1
+\ldots
\big)&
\\
,
D_{x,r} T_{y,r} v + \eps^{-1} D_{y,r} v & \Big\>_{\calM\times\calP}
= \<f, v\>_{\calM\times\calP}
.
\end{split}
\end{equation}

We now expand this equation in powers of $\eps$.
For that, we use the approximation $D_{x,r}\approx \nabla_{x,r}$ (i.e., we essentially use Taylor series to expand $D_{x,r}$), and the notations $V_\eps(\bullet; x) = V(\bullet; y)$ and $\calR_\eps(x) = \calR(y)$, and change a sum over $\calM$ to an integral:
\begin{equation} \label{eq:hom:scales_separated}
\begin{split}
\Big\<
\sum_{r\in\calR} V'_r\big(
\nabla_{x,\calR} u^0+\eps \nabla_{x,\calR} T_{y,\calR} u^1+ D_{y,\calR}u^1+\ldots
\big)&
\\
, \nabla_{x,r} T_{y,r} v + \eps^{-1} D_{y,r} v & \Big\>_{\Omega\times\calP}
= \<f, v\>_{\Omega\times\calP}
,
\end{split}
\end{equation}
where $\<\bullet\>_{\Omega\times \calP}$ is a short-hand for $\int_\Omega \<\bullet\>_\calP \dx$.

We first collect the $O(\eps^{-1})$ terms in \eqref{eq:hom:scales_separated}:
\begin{align*}
\Big\<
\sum_{r\in\calR} V'_r\big(\nabla_{x,\calR} u^0+ D_{y,\calR}u^1\big)
, D_{y,r} v \Big\>_{\Omega\times\calP}
= 0
.
\end{align*}
As usual in homogenization we write the solution of this equation (of course, equipped with the zero-average boundary conditions) as $u^1(x,y) = \chi(\nabla_x u^0(x); y)+\baru^1(x)$, where $\chi=\chi(\mF; y)\,:\,\bbR^{d\times d}\times\calP\to\bbR^d$ solves
\begin{equation} \label{eq:chi_def}
\text{find $\chi(\mF, \bullet)\in\calU_\#(\calP)$ s.t.}\quad
\Big\<
\sum_{r\in\calR} V'_r\big(\mF \calR + D_{y,\calR}\chi(\mF)\big)
, D_{y,r} \sigma \Big\>_{\calP}
= 0
\quad \forall \sigma\in\calU_\#(\calP).
\end{equation}
As earlier, $\chi(\mF, y)$ can be formally understood as some solution to \eqref{eq:chi_def}, similarly to the shift vector function $\bq^h(\mF)$ discussed in Remark \ref{rem:no_discussion_of_existence}.
We will establish the formal equivalence of $\chi(\mF, y)$ and $\bq^h(\mF)$ in Theorem \ref{thm:methods-are-equivalent}, hence the results in the cited references \cite{AbdulleLinShapeev2011_analysis, EMing2007} are applicable to well-posedness of \eqref{eq:chi_def} also.

To obtain the equation for the homogenized solution $u^0(x)$, we collect the $O(\eps^{0})$ terms in \eqref{eq:hom:scales_separated} and use the test function $\bar{v}$ of $x$ only:
\begin{align*}
\Big\<
\sum_{r\in\calR} V'_r\big(
\nabla_{x,\calR} u^0 + D_{y,\calR}u^1
\big)
, \nabla_{x,r} \bar{v} \Big\>_{\Omega\times\calP}
= \<f, \bar{v}\>_{\Omega\times\calP}
.
\end{align*}
This leads to the homogenized equation
\begin{equation} \label{eq:homogenized}
\<\delPhi^0(\nabla_x u^0), \nabla_x \bar{v}\>_\Omega = \<f,\bar{v}\>_\Omega,
\end{equation}
or equivalently, in the strong form $- \nabla_x\cdot \delPhi^0(\nabla_x u^0) = f(x)$, where $\delPhi^0 \,:\,\bbR^{d\times d}\to\bbR^{d\times d}$ satisfies
\begin{equation}
\label{eq:delPhi0-def}
\delPhi^0(\mF) =
\Big\<
\sum_{r\in\calR(y)} V'_r\big(\mF \calR + D_{y,\calR}\chi(\mF)\big) \, r^{\!\top}
\Big\>_{y\in\calP}.
\end{equation}
Thus, we obtained the equation for the homogenized displacement $u^0$ with the homogenized tensor $\delPhi^0$.
Equation \eqref{eq:homogenized} needs to be supplemented with boundary conditions, for instance by requiring that $u^0$ is periodic and has zero average.

As an illustrative example, in the case of a pair interaction potential we can write 
$V(D_{\calR(y)}u;y)=\sum_{r\in \calR(y)}\Phi_r(D_r u;y)$ (cf. the Lennard--Jones potential in (\ref{eq:problem_formulation:LJ})),
consequently,
\begin{displaymath} \delPhi^0(\mF) =
\Big\<
\sum_{r\in\calR(y)} \Phi'_r\big(\mF r + D_{y,r}\chi(\mF)\big) \, r^{\!\top}
\Big\>_{y\in\calP}.
\end{displaymath}

\begin{remark}\label{rem:homogenization:discrete}
In the above formal arguments we assumed, for simplicity, that the external force $f=f(x)$ is non-oscillating (i.e., effectively does not depend on $y$) and is defined on all of $\Omega$.
We emphasize that oscillatory external forces, (of the form $f(x,y)|_{y=x/\eps}$) could also be considered.
The homogenized equation would then depend on a proper average of the external forces.
The assumption that $f$ is defined on the entire $\Omega$ can later be relaxed once the homogenized equations are discretized on a finite element mesh.
\end{remark}

\begin{remark}
Instead of upscaling the original discrete problem \eqref{eq:variational_equation} to a continuous problem of nonlinear elasticity \eqref{eq:homogenized}, one can consider an alternative approach where the upscaled model is discrete.

For instance, one can approximate \eqref{eq:homogenization:extended_discrete} by taking discrete $x\in\calL$ and approximating $D_{x,r}\bullet \approx (D_x\bullet) r$, where $D_x u(x) \in\bbR^{d\times d}$ is the discrete gradient of $u\in\calU_\per(\calL)$ at the point $x\in\calL$ defined as $(D_x u(x))e_k = D_{x,e_k} u(x)$, $k=1,\ldots,d$, $e_k$ is the $k$-th standard basis vector of $\bbR^d$.
Following the described above procedure of asymptotic expansion one can derive the following upscaled equation
\begin{equation} \label{eq:discr_homogenized}
\<\delPhi^0(D_x u^0), D_x v\>_\calL = \<f, v\>_\calL
\quad\forall v\in\calU_\#(\calL)
,
\end{equation}
where $\delPhi^0$ is defined by \eqref{eq:delPhi0-def}, the same equation as for the continuum homogenization.
The equation \eqref{eq:discr_homogenized} is upscaled in the sense that $\delPhi^0$ no longer depends on the fast variable $y$, and we can apply the standard QC to it.
The reader can refer to our preprint \cite{AbdulleLinShapeev2010} for a similar approach.
An advantage of the discrete homogenization is that it is not required to assume a continuous force $f$ to derive \eqref{eq:discr_homogenized}.
\end{remark}

\subsubsection*{Underlying Homogenized Energy}

We claim that, formally, the function $\delPhi^0(\mF)$ defined by \eqref{eq:delPhi0-def} is the derivative of the following function
\begin{equation}\label{eq:Phi0-def}
\Phi^0(\mF) :=
\big\<
V\big(\mF \calR + D_{y,\calR}\chi(\mF)\big)
\big\>_{y\in\calP}
,
\end{equation}
where $\chi=\chi(\mF)$ is some solution to \eqref{eq:chi_def}.

Indeed, assuming enough regularity of $V$ and $\chi$, we can compute the variation of \eqref{eq:Phi0-def} with respect to $\mF$:
\begin{equation}\label{eq:delPhi0-variation}
\delPhi^0(\mF_0) \!:\! \mG
=
\Big\<
\sum_{r\in\calR} V'_r\big(\mF \calR + D_{y,\calR}\chi(\mF_0)\big)
	\cdot (\mG r + D_{y,\calR} \del\chi(\mF) \!:\! \mG)
\Big\>_{y\in\calP}
.
\end{equation}
Since $\del\chi(\mF) \!:\! \mG \in \calU_\#(\calP)$, the second term in \eqref{eq:delPhi0-variation} drops due to \eqref{eq:chi_def} and we have
\[
\delPhi^0(\mF) \!:\! \mG
=
\Big\<
\sum_{r\in\calR} V'_r\big(\mF \calR + D_{y,\calR}\chi(\mF)\big)
	\cdot \mG r
\Big\>_{y\in\calP}
=
\Big\<
\sum_{r\in\calR} V'_r\big(\mF \calR + D_{y,\calR}\chi(\mF)\big) r^{\!\top} \!:\! \mG
\Big\>_{y\in\calP}
,
\]
which is consistent with \eqref{eq:delPhi0-def}.

Hence, the equations \eqref{eq:homogenized} can be written as
\[
\<\delE^0(u^0), v\>_\Omega = \<f, v\>_\Omega,
\]
where
\begin{equation}
\label{eq:homogenization:generalizations:E}
E^0(u^0)
:=
\int_\Omega \Phi^0(\nabla u^0)\dx
.
\end{equation}
The fact that the homogenized equations have an underlying energy may be important in some applications where, for instance, one chooses to use nonlinear conjugate gradient algorithms or needs to check for stability of numerical solutions.

\subsection{Application of Homogenization to the Simplified Model}\label{sec:homogenization:simplfied}

In order to make the steps of the above formal homogenization technique more transparent, we apply it to the simplified model \eqref{eq:problem_1d_linear}, written in a strong form as
\[
	D_{-r} (\psi_\eps(x) D_r u(x)) = f(x)
	\quad \forall x\in\calM
	,
\]
where $\psi(y) := \psi_\eps(\eps y)$ and $r:=\smfrac1m$ is fixed throughout this subsection.
We calculate the application of the full derivative $D_r$ (see Appendix \ref{sec:notations:operators} for the precise definition) to \eqref{equ:asympt_exp}:
\begin{align*}
D_r u(x,y)
=~&
(D_{x,r} T_{y,r} + \smfrac1\eps D_{y,r}) (u^0(x) + \eps u^1(x,y) + \ldots)
\\ =~&
D_{x,r} u^0(x) + D_{y,r} u^1(x,y) + \eps D_{x,r} T_{y,r} u^1(x,y) + \ldots
\end{align*}
and hence insert \eqref{equ:asympt_exp} into \eqref{eq:variational_equation}:
\[
	(D_{x,-r} T_{y,-r} + \smfrac1\eps D_{y,-r}) \big(
		\psi(y) (D_{x,r} u^0(x) + D_{y,r} u^1(x,y) + \ldots)
	\big) = f(x)
	\quad \forall x\in\calM,~\forall y\in\calP
,
\]
and change the discrete derivative with respect to $x\in\calM$ to the continuum derivative with respect to $x\in\bbR$:
\[
	(\nabla_{x,-r} T_{y,-r} + \smfrac1\eps D_{y,-r}) \big(
		\psi(y) (\nabla_{x,r} u^0(x) + D_{y,r} u^1(x,y) + \ldots)
	\big) = f(x)
	\quad \forall x\in\bbR,~\forall y\in\calP
.
\]

Collecting the $O(\eps^{-1})$ terms in this equation yields
\[
	D_{y,-r} \big(
		\psi(y) (\nabla_{x,r} u^0(x) + D_{y,r} u^1(x,y))
	\big) = f(x)
	\quad \forall x\in\bbR,~\forall y\in\calP
\]
and we can formally write the solution to this equation as
\[
u^1(x,y) = \chi(\nabla_x u^0(x)),
\]
where
the cell problem \eqref{eq:chi_def} in the strong form reads
\[
D_{y,-r} \big(\psi\, (Fr + D_{y,r} \chi(F))\big) = 0.
\]

For our simplified model, the cell problem admits the exact solution
\begin{equation}
D_{y,r} \chi(F) = \frac{C}\psi - Fr.
\label{eq:linear-nn_exact_DYchi}
\end{equation}
with $C = Fr\,\<1/\psi\>_\calP$, and $\calP$ given by \eqref{eq:simplified_model:P}.
The homogenized energy density is therefore
\begin{equation}
\Phi^0(F) = \big\<\psi \smfrac12 (Fr + D_{y,r} \chi)^2\big\>_\calP
=
\Big\<\psi \smfrac12 \big(Fr + \smfrac{C}\psi - Fr\big)^2\Big\>_\calP
= \smfrac12\,\frac{C^2}{\<1/\psi\>_\calP}
=
\<1/\psi\>_\calP^{-1}\,\frac{(Fr)^2}2,
\label{eq:linear-nn_exact:homogenized-tensor}
\end{equation}
which yields the homogenized energy
\[
E^0(u^0) = \int_0^1 \<1/\psi\>_\calP^{-1} \frac{(\nabla_r u^0)^2}2\, \dx
,
\]
with the correct form of the effective spring constant $\psi^0 = \<1/\psi\>_\calP^{-1}$, given by the harmonic average of $\psi$.

We note that this homogenization procedure, and the result that the effective energy density coefficient is in the form of a harmonic mean of the original coefficient, are well known for PDEs \cite[Chap.\ 1]{BensoussanLionsPapanicolaou1978}.
The expression for $\psi^0$ is also in agreement with the one found in Section \ref{sec:QC-heterogeneous:failure} for $m=2$.

\section{Homogenized QC}\label{sec:HQC}

We formulate a numerical macro-to-micro method for treating multilattices, which we call the homogenized quasicontinuum method (HQC).
We introduce HQC in the framework of numerical homogenization. For the case of materials with known periodic structure (i.e., crystalline materials)
the HQC method will be shown to be equivalent to applying finite elements to the homogenized equations (see Theorem \ref{thm:methods-are-equivalent}).

We emphasize that HQC can be generalized to non-crystalline materials and to time-dependent zero-temperature and, possibly, finite-temperature problems.
Indeed, in Section \ref{sec:stochastic} we give an application of HQC to a stochastic material and in Section \ref{sec:unsteady} we present an application of HQC to a 1D time-dependent zero-temperature evolution.
In addition, the HQC serves a convenient framework for the error analysis \cite{AbdulleLinShapeev2011_analysis, AbdulleLinShapeev2010}.

We present the HQC algorithm assuming that the microstructure is a function of the macroscopic displacement.
A reformulation analogous to the concurrent coupling of \cite{SorkinElliottTadmor} is also possible (cf.\ also Remark \ref{rem:concurrent_coupling}).

\subsection{HQC Method}\label{sec:HQC:method}
The method will be presented using macro-to-micro framework as used in some numerical homogenization procedures
\cite{Abdulle2009, EEL2007, GeersKouznetsovaBrekelmans2010, MieheBayreuther2007, TeradaKikuchi2001}.
We present the method for the case when the external force $f=f_\eps$ may be microstructure-dependent.

\subsubsection{Macroscopic affine displacement}
We again assume a partition $\calTh$ of the domain $\Omega$ into simplicial elements $T$, recall the definition of the space $\calU^h_\per$, \eqref{eq:UH-space}, and introduce its subspace of zero-mean functions $\calU^h_\#\subset\calU^h_\per$.

\subsubsection{Sampling Domains}\label{sec:HQC:method:sampling_domains}
We choose a representative position $x_T^\rep\in\calL$ and a sampling domain $S_T^\rep := x_T^\rep + \eps\calP$ associated with each $T\in\calTh$.
The sampling domain is normally chosen inside $T$ (the mesh can be highly refined in certain regions and therefore some sampling domains $S_T^\rep$ may be bigger than $T$).

The sampling domains have the associated operator of averaging over the sampling domain, $\<\bullet\>_{x\in S_T^\rep}$ and the functional space $\calU_\#(S_T^\rep) = \calU_\#(\eps\calP)$ (see \eqref{eq:Uper} for the precise definition).

\subsubsection{Energy and Macro Nonlinear Form}

Define the atomistic interaction energy of the HQC method
\begin{equation}
\label{eq:Ehqc_def}
E^\hqc(u^h)
:=
\sum_{T\in\calTh} |T|
\big\<V_\eps(D_{\calR_\eps} R_T(u^h))\big\>_{x\in S_T^\rep},
\end{equation}
where
$R_T(u^h)$, defined by \eqref{eq:HQC:microproblem}, is the microfunction constrained by $u^h$ in the sampling domain $S_T^\rep$.

The functional derivative of the above energy reads
\begin{equation}
\<\delE^\hqc(u^h), v^h\>_\Omega
=
\sum_{T\in\calTh} |T|
\Big\<\sum_{r\in\calR_\eps}V'_{\eps,r}(D_{\calR_\eps} R_T(u^h)), D_r\, \delR_T(u^h)\,v^h\Big\>_{x\in S_T^\rep},
\label{eq:HQC:bilinear-form}
\end{equation}
where $\delR_T(u^h)$ is the functional derivative of the reconstruction $R_T(u^h)$ defined below.

\subsubsection{Microproblem}\label{HQC:nonlinear:cell}
Given a function $u^h\in \calU_\per^h,$ $R_T(u^h)$
is a function such that $R_T(u^h)-u^h_\lin\in \calU_\#(S_T^\rep)$
and
\begin{equation}\label{eq:HQC:microproblem}
\Big\<\sum_{r\in\calR_\eps}V'_{\eps,r}(D_{\calR_\eps} R_T(u^h)),~ D_rs\Big\>_{x\in S_T^\rep}
=0
\quad \forall s\in \calU_\#(S_T^\rep)
,
\end{equation}
where $u^h_\lin$ is an affine extrapolation of $u^h|_T$ over the entire $\bbR^d$.
If $S_T^\rep \subset T$ then $u^h_\lin$ can be substituted with $u^h$.

\begin{remark}\label{rem:nonlinear_reconstruction_stable_equilibrium}
	When modeling essentially nonlinear phenomena (e.g., martensite-austenite phase transformation), one should require that the microstructure corresponds to a stable equilibrium.
	That is, one should require, in addition to \eqref{eq:HQC:microproblem}, that $w=R_T(u^h)-u^h_\lin\in \calU_\#(S_T^\rep)$ is a local minimum of $\<V_\eps(D_{\calR_\eps} (u^h_\lin+w))\>_{x\in S_T^\rep}$ \cite[p.\ 238]{TadmorSmithBernsteinEtAl1999}.
\end{remark}

\begin{remark}\label{rem:linear_reconstruction}
In the case of linear interaction, the reconstruction $R_T$ is a linear function and hence $\delR_T(u^h)v^h = R_T(v^h)$, which makes the derivative of the HQC energy \eqref{eq:HQC:bilinear-form} take the form
\[
\<\delE^\hqc(u^h), v^h\>_\Omega
=
\sum_{T\in\calTh} |T|
\Big\<\sum_{r\in\calR_\eps}V'_{\eps,r}(D_{\calR_\eps} R_T(u^h)), D_r R_T(v^h)\Big\>_{x\in S_T^\rep}.
\]
\end{remark}

\begin{remark}\label{rem:simplify-the-bilinear-form}
The functional derivative of the HQC energy \eqref{eq:HQC:bilinear-form} can equivalently be written as
\begin{equation}
\<\delE^\hqc(u^h), v^h\>_\Omega
=
\sum_{T\in\calTh} |T|
\Big\<\sum_{r\in\calR_\eps}V'_{\eps,r}(D_{\calR_\eps} R_T(u^h)), (\nabla_r v^h|_T)\Big\>_{x\in S_T^\rep},
\label{eq:HQC:bilinear-form_2}
\end{equation}
by noting that $D_r \delR_T(u^h)v^h = D_r v^h_\lin+\big(D_r \delR_T(u^h)v^h)-D_r v^h_\lin\big)$,
that
\[
\sum_{r\in\calR_\eps} \<V'_{\eps,r}(D_{\calR_\eps} R_T(u^h)),~ (D_r \delR_T(u^h)v^h-D_r v^h_\lin)\>_{x\in S_T^\rep}
=0,
\]
in view of \eqref{eq:HQC:microproblem},
and that $D_r v^h_\lin = \nabla_r v^h$ on each $T$.
Here we used the fact that $\delR_T(u^h)v^h-v^h_\lin \in \calU_\#(S_T^\rep)$ which follows from taking the functional derivative of $R_T(u^h)-u^h_\lin \in \calU_\#(S_T^\rep)$.
\end{remark}

\subsubsection{Reconstruction}\label{HQC:nonlinear:reconstruction}
The functions $R_T(u^h)$ describe the microstructure of the solution inside each $S_T^\rep$.
One can reconstruct the solution describing the microstructure, $u^{h,\c}$, from the homogenized solution $u^h$ by combining $R_T(u^h)$ into a single function defined on the entire atomistic lattice $\calM$:
\begin{equation}
\label{eq:HQC:periodic-extension}
u^{h,\c}(x) = R_T(u^h)(x)
\quad (x\in T\cap\calM)
.
\end{equation}
That is, we effectively extend $R_T(u^h)$ periodically on each $T$.
It should be noted that \eqref{eq:HQC:periodic-extension} does not uniquely determine $u^{h,\c}(x)$ if $x\in\partial T$ for some $T\in\calTh$.

\subsubsection{Variational Problem}
We define the homogenized quasicontinuum approximation as the solution $u^h\in \calU_{\#}^h$ of
\begin{equation}
\<\delE^\hqc(u^h), v^h\>_\Omega
=
F^\hqc(v^h)
\quad\forall v^h\in \calU_{\#}^h
\label{eq:HQC:problem}
\end{equation}
where
\begin{equation}
F^\hqc(v^h) =
\sum_{T\in\calTh} |T| \<f_\eps, v^h\>_{x\in S_T^\rep}
.
\label{eq:HQC:RHS}
\end{equation}
If the external force is smooth, it could instead be evaluated for a single representative atom.

In the case of linear nearest-neighbor 1D interaction it can be shown that (5.7) is well-posed and that the homogenized quasicontinuum solution $u^h$ approximates the solution $u$ of the original equations only in the $L^2$-norm. To get a good approximation in the $H^1$-norm, the reconstructed solution $u^{h,\c}$ should instead be considered. (This is analogous to the case of continuum homogenization, see discussion in Section \ref{sec:problem_formulation:simplified}.) We will report the analysis for the nonlinear case in a separate paper (see the preprint \cite[Theorems 4 and 5]{AbdulleLinShapeev2010} for the analysis of a linear model).

\subsection{HQC Algorithm}\label{sec:HQC:algorithm}

The problem \eqref{eq:HQC:problem} is nonlinear, and its practical implementation is usually done by Newton's method.
We briefly sketch below an algorithm for solving \eqref{eq:HQC:problem}.

For Newton's method we need to compute the second derivative of the energy \eqref{eq:Ehqc_def}:
\begin{equation} \label{eq:HQC:second-variation}
\<\ddelE^\hqc(u^h) w^h, v^h\>_\Omega
=
\sum_{T\in\calTh} |T|
\bigg\<\sum_{r,\rho\in\calR_\eps} V''_{\eps,r,\rho}(D_{\calR_\eps} R_T(u^h)) D_\rho \delR_T(u^h) w^h
	,~ D_r \delR_T(u^h) v^h
\bigg\>_{x\in S_T^\rep}
.
\end{equation}

\subsubsection{Newton's Iterations for the Macroproblem}

The algorithm based on Newton's method consists of choosing an initial guess $u^{h,(0)}\in \calU^h_{\#}$ and performing iterations
\begin{equation}
	\big\<\ddelE^\hqc\big(u^{h,(n)}\big) \big(u^{h,(n+1)} - u^{h,(n)}\big), v^h\big\>_\Omega
	=
	\big\<\delE^\hqc\big(u^{h,(n)}\big), v^h\big\>_\Omega
	+
	F^\hqc(v^h)
	\quad \forall v^h\in \calU_{\#}^h,
\label{eq:HQC:Newton-macro-iterations}
\end{equation}
with $n=0,1,\ldots$, until $u^{h,(n+1)}$ becomes close to $u^{h,(n)}$ in a chosen norm.

To solve the linear system \eqref{eq:HQC:Newton-macro-iterations} for $u^{h,(n+1)} - u^{h,(n)}\in \calU_{\#}^h$, we choose a nodal basis $w_k^h$ ($1\le k\le K$) of $\calU_\per^h$.
One way to satisfy the condition $\<u^h\>_\Omega=0$ would be to perform all the computations with one basis function eliminated (e.g., to consider $w_k^h$ for $2\le k\le K$), and post-process the final solution as $u^h - \<u^h\>_\Omega$.

The stiffness matrix of the system \eqref{eq:HQC:Newton-macro-iterations} will thus be
\[
A_{lm}
=
\big\<\ddelE^\hqc\big(u^{h,(n)}\big) w_l^h, w_m^h\big\>_\Omega
\]
and the load vector will be
\[
b_m
=
\big\<\delE^\hqc\big(u^{h,(n)}\big), w_m^h\big\>_\Omega
+
F^\hqc(w_m^h).
\]
As given by the formula \eqref{eq:HQC:second-variation} we need to compute the solution of microproblem $R_T\big(u^{h,(n)}\big)$ on each sampling domain $S_T^\rep$ as well as its derivative $\delR_T\big(u^{h,(n)}\big)w^h_l$.

\subsubsection{Solution of the Microproblem}

The microproblem \eqref{eq:HQC:microproblem} can also be solved with Newton's method.
For that, in each $T$ one needs to choose an initial guess $R_T^{(0)}$ to $R_T(u^{h,(n)})$, for instance $R_T^{(0)}(x) := u^{h,(n)}(x)$, and solve
\begin{align*}
\Big\<
	\sum_{r\in\calR_\eps}V'_{\eps,r}\big(D_{\calR_\eps} R_T^{(\nu)}\big)
	+
	\sum_{r,\rho\in\calR_\eps}V''_{\eps,r,\rho}\big(D_{\calR_\eps} R_T^{(\nu)}\big)
	D_\rho \big(R_T^{(\nu+1)}-R_T^{(\nu)}\big)
	,~ D_rs\Big\>_{x\in S_T^\rep}
=0
\\
\forall s\in \calU_\#(S_T^\rep)
,
\end{align*}
with respect to $R_T^{(\nu+1)}$ ($\nu=0,1,\ldots$) constrained by $R_T^{(\nu+1)} - u^{h,(n)}_\lin\in \calU_\#(S_T^\rep)$, until the difference between $R_T^{(\nu+1)}$ and $R_T^{(\nu)}$ is small in a chosen norm.

After that, we can compute $\delR_T w^h_l = \delR_T\big(u^{h,(n)}\big)w^h_l$ by solving
\begin{equation}
\Big\<
	\sum_{r,\rho\in\calR_\eps}V''_{\eps,r,\rho}\big(D_{\calR_\eps} R_T^{(\nu)}\big)
	D_\rho (\delR_T w^h_l)
	,~ D_rs\Big\>_{x\in S_T^\rep}
=0
\quad \forall s\in \calU_\#(S_T^\rep)
\label{eq:HQC:equation-for-reconstruction-variation}
\end{equation}
constrained by $\delR_T w^h_l - (w^h_l)_\lin\in \calU_\#(S_T^\rep)$.
Notice that the gradients of all but $d+1$ basis functions $D_r (w^h_l)_\lin$ inside $T$ are zero, which implies that we essentially need to solve the problem \eqref{eq:HQC:equation-for-reconstruction-variation} $d+1$ times.

Also observe that when computing $\delR_T\big(u^{h,(n)}\big)w^h_l$, we need to invert the same linear operator as in the final Newton iteration, which allows for some additional optimization.

\subsubsection{Possible Modifications of the Algorithm}\label{sec:HQC:algorithm:modifications}

First, notice that when solving for $u^{h,(n+1)}$ we could linearize the problem on the previous iteration $u^{h,(n)}$.
In that case we would have linear cell problems and thus we would need only outer Newton iteration, but it would be required to keep the values of the micro-solution $R_T(u^{h,(n)})$ from the previous iteration.
We notice, however, that for a practical implementation of the above algorithm it may also be required to keep the values of the micro-solution: one needs these values to initialize the inner Newton iterations; depending on the initial guess for the microproblem the iterations may converge to a wrong microstructure.

Another modification could be to compute the contribution of the external force $f_\eps$ in \eqref{eq:HQC:RHS} for a single atom in the case of no oscillations in $f_\eps$.

In the case of linear interaction, the algorithm becomes simpler: one does not need to do Newton iterations.
Nevertheless, the algorithm in Section \ref{sec:HQC:algorithm} is applicable to the linear problem where it converges in just one iteration.

\section{Equivalence of Numerical Methods for Multilattices} \label{sec:equivalence}

In this section we show the equivalence of three different methods for computing equilibrium of multilattice crystals, namely (1) the proposed HQC method, (2) finite element discretization of continuum homogenization, and (3) MQC. We only compare the interaction energy of the method, since the external forces for these methods can always be chosen same.

Below we specify the three methods that we compare.
It should be noted that given the macroscopic displacement $u^h$ we cannot guarantee uniqueness of the energy as there may be several solutions to the micro-problems corresponding to different phases of a multilattice crystal.
To rigorously address such non-uniqueness, we allow for all possible combinations of microfunctions in each element $T\in\calT_h$, and compare the set of the resulting energies on a fixed $u^h\in\calU^h_\per$ for the three methods.

In the following definitions we adopt the convention that for two sets, $A$ and $B$, and a number, $\gamma$, $A+B := \{a+b : a\in A, b\in B\}$ and $\gamma A := \{\gamma a:a\in A\}$.

\begin{description}
\item[Method 1. (HQC) ] For $u^h\in\calU^h_\per$ we define the energy of the HQC method as a set $E^\hqc(u^h)\subset\bbR$,
\begin{equation}\label{eq:equivalence:Ehqc}
E^\hqc(u^h) := \sum_{T\in\calTh} |T| \, e^\hqc_T(u^h),
\end{equation}
where $e^\hqc_T(u^h)\subset\bbR$ is defined as
\[
e^\hqc_T(u^h) := \big\{
	\big\<V_\eps(D_{\calR_\eps} R_T(u^h))\big\>_{x\in S_T^\rep}
	: R_T(u^h) \text{ is a solution to \eqref{eq:HQC:microproblem}}
\big\}
.
\]

\item[Method 2. (FEM for homogenized equations) ]
The energy of FEM discretization of the homogenized energy is $E^0(u^h)$, defined by
\begin{equation}\label{eq:equivalence:Ezero}
E^0(u^h) := \sum_{T\in\calTh} |T| \, \Phi^0(\nabla u^h|_T),
\end{equation}
where $\Phi^0$ is defined as a set
\[
\Phi^0(\mF) :=
\big\{
	\big\<
	V\big(\mF \calR + D_{y,\calR}\chi\big)
	\big\>_{y\in\calP}
	: \chi \text{ is a solution to \eqref{eq:chi_def}}
\big\}
.
\]

\item[Method 3. (Multilattice QC)]
We define
\begin{equation}\label{eq:equivalence:Emqc}
E^\mqc(u^h) := \sum_{T\in\calTh} |T| \, e^\mqc_T(u^h),
\end{equation}
where
\begin{align*}
e^\mqc_T(u^h) := \Big\{
	&
	\frac1m \sum_{\beta=0}^{m-1}
	V_{\eps,\beta}\Big(
		\big(\nabla u^h|_T\big)\calR_{\eps,\beta} + \sum_{\alpha=1}^{m-1} \big(q^h_\alpha|_T\big) D_{\calR_{\eps,\beta}} w_\alpha(\eps p_\beta)
\Big)
	\\&
	: \bq^h = \bq(\nabla u^h) \text{ is a solution to \eqref{eq:QC_general_shift-vectors-equation}}
\Big\}.
\end{align*}
\end{description}

\begin{theorem}\label{thm:methods-are-equivalent}
Let $\calT_h$ be a triangulation of $\Omega$ and $\calU^h_\per$ be the associate function space defined by \eqref{eq:UH-space}.
Then for any $u^h\in\calU^h_\per$, there holds
\[
E^\hqc(u^h) = E^0(u^h) = E^\mqc(u^h),
\]
where $E^\hqc(u^h)$, $E^0(u^h)$, $E^\mqc(u^h)$ are defined by, respectively, \eqref{eq:equivalence:Ehqc}, \eqref{eq:equivalence:Ezero}, \eqref{eq:equivalence:Emqc}.
\end{theorem}

\begin{proof}
{\it Part 1, $E^\hqc(u^h) = E^0(u^h)$.}
First, we show that the micro-functions of Methods 1 and 2, $R_T$ and $\chi$, are related through
\begin{equation}\label{eq:methods-are-equivalent:one_two_microfunctions}
\big(R_T(u^h)\big)(x) = u^h_\lin(x) + \eps \chi\big(\nabla u^h|_T;\smfrac x\eps\big).
\end{equation}
Indeed, denote $\mF = \nabla u^h|_T$ and compute $D_r R_T(u^h)$:
\begin{equation}
\label{eq:equiv_thm_DrRk}
D_r R_T(u^h)
= D_r u^h_\lin + \eps D_r \chi\big(\mF;\smfrac x\eps\big)
= \mF r + D_{y,r} \chi\big(\mF;\smfrac x\eps\big)
.
\end{equation}
The following calculation shows that the left-hand sides of \eqref{eq:HQC:microproblem} and \eqref{eq:chi_def} coincide up to a factor $\eps^{-1}$:
\begin{align*}
\Big\<\sum_{r\in\calR} V'_{\eps,r}(D_{\calR_\eps(x)} R_T(u^h); x), D_r s(x)\Big\>_{x\in S_T^\rep}
=~&
\Big\<\sum_{r\in\calR} V'_r(D_{\calR(y)} R_T(u^h); y), \eps^{-1} D_{y,r} s(\eps y)\Big\>_{y\in \calP}
\\ =~&
\eps^{-1} \Big\<\sum_{r\in\calR} V'_r(\mF \calR + D_{y,\calR}\, \chi(\mF;y); y), D_{y,r} \sigma(y)\Big\>_{y\in \calP}
,
\end{align*}
where we do the change of the independent variable $y=\smfrac x\eps$, and of the test function $\sigma(y) = s(\eps y)$.
Hence \eqref{eq:methods-are-equivalent:one_two_microfunctions} indeed relates the set of solutions of \eqref{eq:HQC:microproblem} and \eqref{eq:chi_def} with $\mF = \nabla u^h|_T$.

The following straightforward calculation concludes the proof of $E^\hqc(u^h) = E^0(u^h)$:
\begin{align*}
e^\hqc_T(u^h)
=~&
\big\<V_\eps \big( D_{\calR_\eps} R_T(u^h)\big)\big\>_{x\in S_T^\rep}
\\ =~&
\big\<V_\eps \big( (\nabla u^h|_T)\calR_\eps + D_{y,\calR_\eps} \chi\big(\nabla u^h|_T;\smfrac x\eps\big)\big)\big\>_{x\in S_T^\rep}
\\ =~&
\big\<V \big( (\nabla u^h|_T)\calR + D_{y,\calR} \chi(\nabla u^h|_T; y)\big)\big\>_{y\in \calP}
=
\Phi^0(\nabla u^h|_T)
\end{align*}
where we used \eqref{eq:equiv_thm_DrRk} in the first step of this calculation.

{\it Part 2, $E^\hqc(u^h) = E^\mqc(u^h)$.}
The main component of the proof consists of fixing $T\in\calTh$ and showing that $q^h_\alpha|_T$ and $R_T(u^h)$ are related through
\begin{equation}\label{eq:methods-are-equivalent:one_three_microfunctions}
q^h_\alpha|_T = U(\eps p_{\alpha}) - U(0)
,
\qquad \alpha=0,\ldots,m-1
,
\end{equation}
where $U := R_T(u^h)-u^h_\lin \in \calU_\#(S_T^\rep)$.

First, assume that $R_T(u^h)$ is a solution to \eqref{eq:HQC:microproblem}.
Notice that due to $\eps\calP$-periodicity of $U$, we can write
\[
U(x) = \sum_{\alpha=0}^{m-1} U(\eps p_\alpha)\, w_\alpha(x),
\]
subtracting the constant $U(0)$ and applying $D_r$ yields
\begin{align*}
D_r U(x)
=~&
D_r \Big(-U(0)+\sum_{\alpha=0}^{m-1} U(\eps p_\alpha)\, w_\alpha(x)\Big)
\\ =~&
D_r \Big(\sum_{\alpha=1}^{m-1} U(\eps p_\alpha)\, w_\alpha(x)\Big)
\\ =~&
D_r\,\sum_{\alpha=1}^{m-1} (\tilde{q}^h_\alpha|_T) w_\alpha(x)
,
\end{align*}
where we used the identity $\sum_{\alpha=0}^{m-1} w_\alpha(x)=1$ for all $x\in\calM$.

We then substitute $q^h_\alpha|_T = U(\eps p_{\alpha}) - U(0)$ into \eqref{eq:QC_general_shift-vectors-equation}.
The argument of $V_\eps$ in \eqref{eq:QC_general_shift-vectors-equation} can be written as
\begin{equation}
\label{eq:Dq_eq_DR}
\begin{split}
		\big(\nabla u^h|_T\big)\calR_{\eps,\beta} + \sum_{\alpha=1}^{m-1} \big(q^h_\alpha|_T\big) D_{\calR_{\eps,\beta}} w_\alpha(\eps p_\beta)
=~&
D_{\calR_{\eps,\beta}} \Big(u^h_\lin + \sum_{\alpha=1}^{m-1} (q^h_\alpha|_T) w_\alpha(\eps p_\beta)\Big)
\\=~&
D_{\calR_{\eps,\beta}} \Big(u^h_\lin + U(\eps p_\beta)\Big)
=
D_{\calR_\eps(x)} R_T(u^h)(x)\big|_{x=\eps p_\beta}
\end{split}
\end{equation}
and therefore, upon noticing that summations over $x=\eps p_\beta$ and over $x\in S_T^\rep$ coincide for the $\eps\calP$-periodic functions, we conclude that the left-hand sides of \eqref{eq:QC_general_shift-vectors-equation} and \eqref{eq:HQC:microproblem} coincide when $s(x)$ is chosen as $s(x) = w_\gamma(x)-\<w_\gamma(x)\>_{x\in\eps\calP}$, $\gamma=1\ldots,m-1$ (then $D_r s = D_r w_\gamma$).
This proves that $q^h_\alpha|_T = U(\eps p_{\alpha}) - U(0)$ satisfies \eqref{eq:QC_general_shift-vectors-equation}.

To show the converse, assume that $\bq^h$ is a solution to \eqref{eq:QC_general_shift-vectors-equation} and let $R_T$ be defined through \eqref{eq:methods-are-equivalent:one_three_microfunctions}.
We then notice that, due to calculation \eqref{eq:Dq_eq_DR}, \eqref{eq:HQC:microproblem} holds with the function $s(x) = w_\gamma(x)$, $\gamma=1\ldots,m-1$ and, obviously, with the function $s(x)=1$. These functions form a basis of $\calU_\per(\calP)=\calU_\per(S_T^\rep)$, therefore \eqref{eq:HQC:microproblem} holds with any $s\in\calU_\#(S_T^\rep)\subset\calU_\per(S_T^\rep)$, that is, $R_T$ is a solution to \eqref{eq:HQC:microproblem}.
This concludes the proof that the set of solutions of \eqref{eq:QC_general_shift-vectors-equation} and \eqref{eq:HQC:microproblem} are related through \eqref{eq:methods-are-equivalent:one_three_microfunctions}.

The stated identity $E^\mqc(u^h)=E^\hqc(u^h)$ follows from $e^\mqc_T(u^h)=e^\hqc_T(u^h)$ which follows directly from \eqref{eq:Dq_eq_DR}.
\end{proof}

\begin{remark}\label{rem:QC_for_discr_hom}
One can consider yet another approach to coarse-graining multilattices, namely consider the discretely homogenized equation \eqref{eq:discr_homogenized} and apply the standard QC method (see Section \ref{sec:QC:homogeneous}) to it.
As a result one will obtain energy of $\<\Phi^0(\nabla u^h)\>_\Omega$ which obviously coincides with the energy of FEM applied to the continuously homogenized equations.
\end{remark}

\begin{figure}
\begin{center}
	\includegraphics{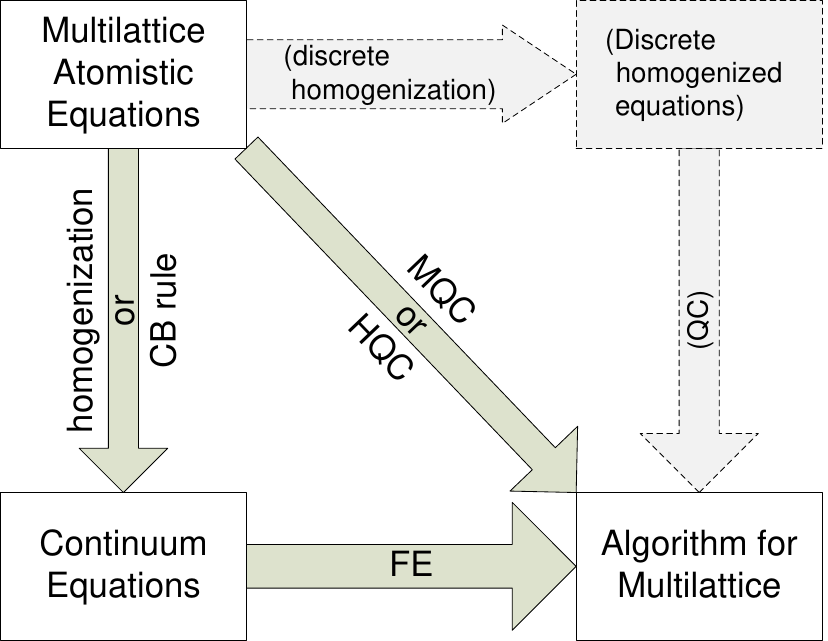}
\end{center}
\caption{Equivalence of different methods. In Theorem \ref{thm:methods-are-equivalent} we prove equivalence of the proposed method (HQC), the multilattice quasicontinuum method \cite{TadmorSmithBernsteinEtAl1999} (MQC), and FEM applied to the homogenized equations.
Also, in Remark \ref{rem:QC_for_discr_hom} we mention that they are equivalent to quasicontinuum method (QC) applied to the discretely homogenized equations.
}
\label{fig:methods-are-equivalent}
\end{figure}

As a corollary of Theorem \ref{thm:methods-are-equivalent} and Remark \ref{rem:QC_for_discr_hom}, the solutions corresponding to the different methods considered, being critical points of the energy, also coincide (of course, provided that the external force is treated in the same way for these methods).
Theorem \ref{thm:methods-are-equivalent} and Remark \ref{rem:QC_for_discr_hom} are graphically summarized in Figure \ref{fig:methods-are-equivalent}.

\section{Application of HQC to a Multilattice}\label{sec:multilattice}

In this section we briefly report the results of application of HQC to the multilattice \cite[Section 8]{AbdulleLinShapeev2010}.
Note that due to Theorem \ref{thm:methods-are-equivalent}, application of MQC to the multilattice gives the same results.

We apply HQC to the 1D linear model problem, same as the one in Section \ref{sec:problem_formulation:simplified} but with a larger interaction range $\calR$.
We compute the HQC solution, $u^h$ and the reconstructed (corrected) solution $u^{h,\c}$, and compare it to the exact solution $u$.

\begin{figure}
\begin{center}
	\includegraphics{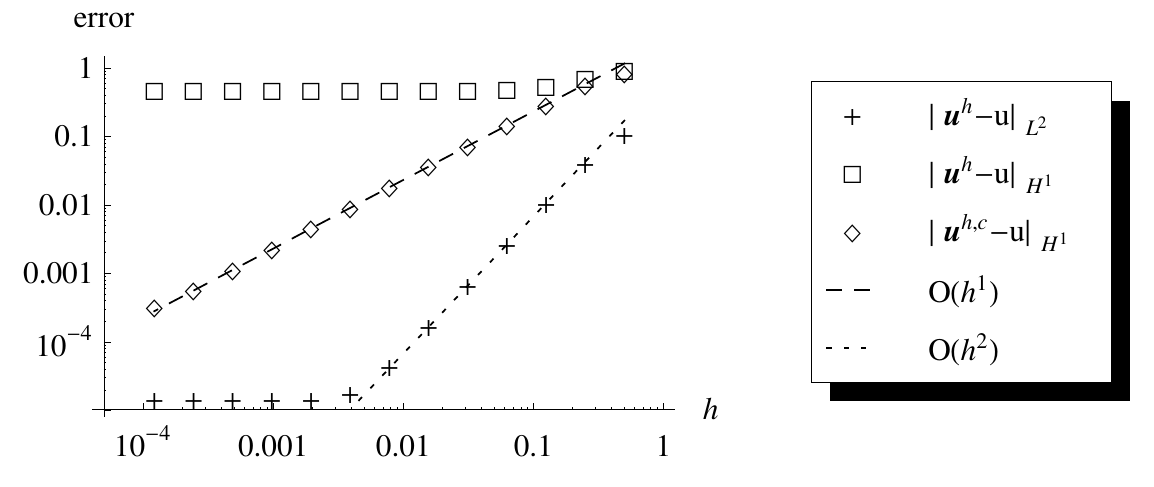}
\end{center}
\caption{Results of application of HQC to a multilattice.
A 1D linear nearest-neighbor interaction model was used.
We observe a first-order convergence of $\|u^{h,\c}-u\|_{H^1(\calM)}$, no convergence of $\|u^h-u\|_{H^1(\calM)}$, and a second-order convergence of $\|u^h-u\|_{L^2(\calM)}$ which stagnates at some point as $h$ is refined.
}
\label{fig:results_multilattice_linear}
\end{figure}

We prove (for nearest-neighbor interaction) and observe in numerical experiments that $\|u^{h,\c}-u\|_{H^1(\calM)}$ converges with the first order in $h$, where $h=\max_{T\in\calT} {\rm diam(T)}$ and $\|\bullet\|_{H^1(\calM)}$ denotes the discrete $H^1$-norm on the lattice $\calM$.
Furthermore, we show that $\|u^h-u\|_{L^2(\calM)} \leq C_1 h^2 + C_2 \eps$, that is, the $L^2$-error converges with the second order up to some point where it stagnates at the level of $C_2\eps$ as $h$ is further refined.
The $H^1$-error of $u^h-u$, on the other hand, stays essentially constant as $h$ is refined.
The results of our numerical experiments are shown in Fig. \ref{fig:results_multilattice_linear}.

The results of application of HQC to a nonlinear interaction are qualitatively same as the presented results for the linear interaction.

\section{Application of HQC to Stochastic Materials}\label{sec:stochastic}

The HQC method can readily be generalized for non-crystalline materials such as glasses or complex metallic alloys.
For that, lacking the period of the microstructure $\calP$, one needs only to take $S_T^\rep$ large enough to accurately represent the material's microstructure.
In this section we present an example of such computation.

In addition to taking $S_T^\rep$ large enough, one could also average over an ensemble of samples of different microstructures for a given macroscopic displacement gradient $\nabla u^h|_T$ in each element $T$; however, we do not pursue this in the present work.
We refer to \cite{BlancLeLions2007_stochastic_lattices, GloriaOtto2011_variance_estimate} and references therein for theoretical studies of stochastic homogenization of lattice energies.

We take an atomistic system of $2048\times 2048$ atoms.
That is, we choose $\eps=\smfrac1{2048}$ and $\calM = \eps \bbZ^2 \cap [0,1)^2$.
The atomistic bonds are chosen to have quadratic interaction energy,
\[
E(u) = \Big\< \sum_{r\in\calR} \smfrac12 \psi_{\eps,r}(x) |D_r u|^2 \Big\>_{x\in\calM}
\]
with $\calR=\{(1,0),(0,1),(1,1),(-1,1)\}$, as illustrated in Fig.\ \ref{fig:2d_springs_short}.
The bonds' strengths $\psi_{\eps,r}$ are randomly generated with a uniform distribution between $0.5$ and $10$ for $r=(1,0)$ and $r=(0,1)$ (i.e., vertical and horizontal bonds) and between $0.1$ and $5$ for $r=(1,1)$ and $r=(-1,1)$ (i.e., diagonal bonds).
Such choice of $\psi_{\eps,r}$ leads to interaction energy $E(u)$ being a convex function of $u$.
Only a single realization of $\psi_{\eps,r}$ is used for this test.

The external force is chosen as
\[
f(x_1,x_2) = 10 e^{-\cos(\pi x_1)^2-\cos(\pi x_2)^2} \left(\begin{array}{c} \sin(2\pi x_1) \\ \sin(2\pi x_2) \end{array}\right)
-\bar{f},
\]
where $\bar{f}$ is determined so that the average of $f$ is zero.
The equilibrium configuration for a system with $32\times32$ atoms is illustrated in Fig.\ \ref{fig:2d-solution-rand-at-large}.
We stress that we no longer have the period of the microstructure $\calP$, and the associated representation of the energy \eqref{eq:E_alt} which was needed in formulation of the MQC method or applying the formal homogenization techniques.

\begin{figure}
\begin{center}
\hfill
\subfigure[]{\label{fig:2d_springs_short}\includegraphics{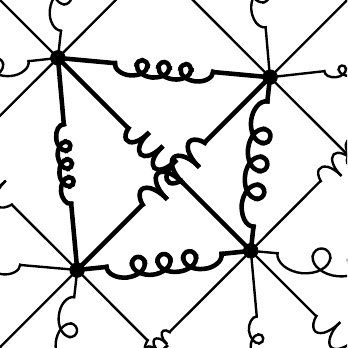}}
\hfill\hfill
\subfigure[]{\label{fig:2d-solution-rand-at-large}\includegraphics{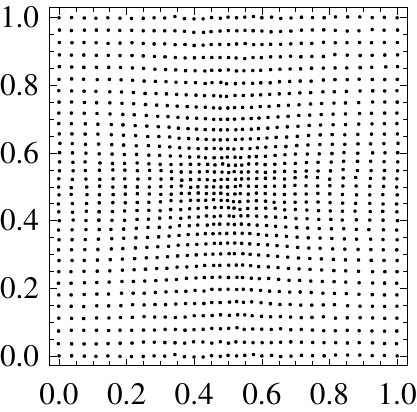}}
\hfill$\mathstrut$
\caption{Stochastic atomistic model: An illustration of the model (left) and an exact solution for $32\times32$ atoms (right).}
\label{fig:2d_springs_illustration}
\end{center}
\end{figure}

We apply the HQC algorithm to the described system.
We choose the sampling domain $S_T^\rep$ as a subsystem of $N_{\rm rep}\times N_{\rm rep}$ atoms.
We then compute the HQC solution and compare it to the exact solution of the problem.
A structured triangular uniform mesh with right-angled triangular elements with the leg size $h=\smfrac14,\smfrac18,\ldots$ is used.

For comparison, we also produce the results of calculation with an affine displacements for computing the effective elasticity tensor in each element $T$; i.e., when atoms are not allowed to relax to equilibrium when an external displacement gradient $\mF$ is applied.

The relative errors of the interaction energy of HQC and affine-displacement solutions ($E^\hqc$ and $E^{\rm ad}$, respectively) as compared to the energy of the exact solution $E$, are plotted in Fig.\ \ref{fig:2d-testcase2-error} for different mesh size $h$ and different sampling domain size $N_{\rm rep}$.
A second-order convergence of HQC and absence of convergence of the solution computed according to the affine deformation can be observed.
One can also see that with $N_{\rm rep}=128$ (and even with $N_{\rm rep}=16$) one can get a rather accurate numerical solution.

\begin{figure}
\begin{center}
\includegraphics{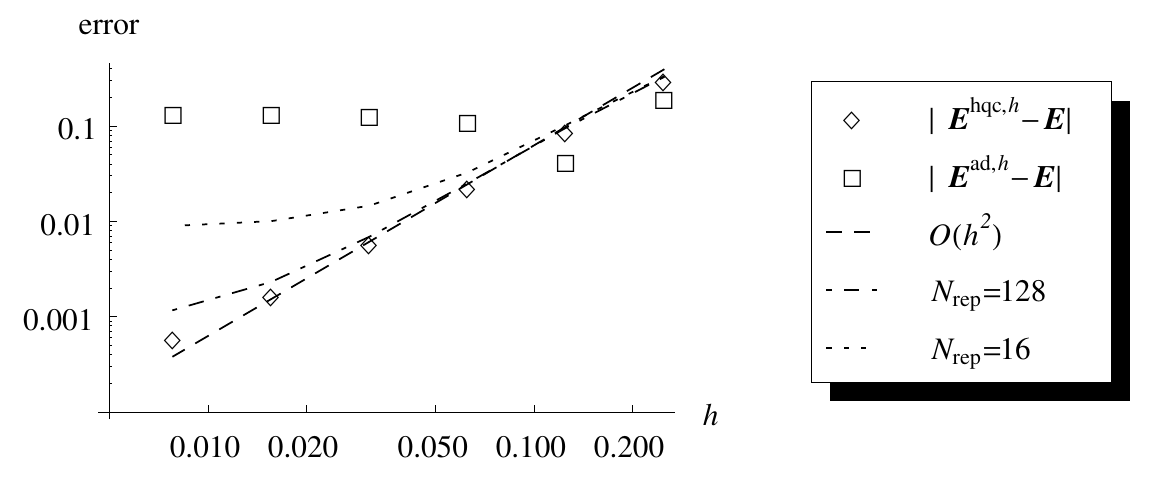}
\end{center}
\caption{Dependence of relative error of computing the energy with HQC and a straightforward application of the Cauchy--Born rule.
The squares and diamonds correspond to $N_{\rm rep}=2048$ (i.e., when the microproblem coincides with the entire system).
A second-order convergence of HQC is observed whereas the Cauchy--Born rule solution does not converge.
The dot-dashed and dotted curves are $|E^\hqc-E|$ for $N_{\rm rep}=128$ and $N_{\rm rep}=16$ respectively.
}
\label{fig:2d-testcase2-error}
\end{figure}

\section{Application of HQC to Time-dependent Problems}\label{sec:unsteady}

We apply the proposed HQC method to the 1D evolution of a multilattice, assumed to be slow (i.e., with no thermal oscillations) described by the following equations
\begin{subeqnarray}
	\< M^\eps \ddot u, v\>_\calM &=& \<\delE(u), v\>_\calM
	\quad \forall v\in \calU_{\per}(\calM)
\\
	u|_{t=0} &=& u^0
\\
	\dot{u}|_{t=0} &=& 0.
\label{eq:unsteady:variational_problem_generic}
\end{subeqnarray}
Here $u=u(t,x) \in C^2([0,T]; \calU_{\per}(\calM))$ is the time-dependent displacement of atom $x$, $u^0 = u^0(x) \in \calU_{\per}(\calM)$ is the initial displacement, $M^\eps(x) = M\big(\smfrac x\eps\big)$ is the mass of atom $x$, $\dot u = \frac{\dd}{\dd t} u$, $\ddot u = \frac{\dd^2}{\dd t^2} u$.
The energy $E(u)$ of a deformation of the multilattice $\calM$ is as defined in Section \ref{sec:problem_formulation:full}.
The masses $M=M(y)$, as well as the interaction, is a $\calP$-periodic function.
We assume no external forces.

One can, assuming no fast oscillations in time of the microstructure, perform the two-scale expansion procedure for the time-depend case (which closely follows the continuum case \cite{BensoussanLionsPapanicolaou1978})
\begin{equation}\label{eq:unsteady:homogenized}
	\< M^0 \ddot u, v\>_\calM =\<\delE^0(u), v\>_\calM
	\quad \forall v\in \calU_{\per}(\calM),
\end{equation}
where $E^0(u)$ is given by \eqref{eq:homogenization:generalizations:E} and $M^0 = \<M\>_\calP$, and likewise formulate the macro-to-micro discretization \cite{AbdulleGrote2011, EngquistHolstRunborg2011}
\[
	\< M^0 \ddot u^h, v^h\>_\calM =\<\delE^\hqc(u^h), v^h\>_\calM
	\quad \forall v\in \calU_{\per}^h.
\]

For the numerical test we take the same lattices as for the simplified model with $m=2$ (see Section \ref{sec:problem_formulation:simplified}).
The atoms interact with the Lennard--Jones potential \eqref{eq:problem_formulation:LJ} with
\[
s_{x,x+\eps r} = \begin{cases}
0.4 & \frac x\eps \text{ is half-integer} \\
1.6 & \frac x\eps \text{ is integer},
\end{cases}
\qquad
\ell_{x,x+\eps r} = \begin{cases}
1.01 & \frac x\eps \text{ is half-integer} \\
0.99 & \frac x\eps \text{ is integer},
\end{cases}
\]
and the cut-off distance $R=3$.
The masses of atoms are
\[
M^\eps(x) = \begin{cases}
1 & \frac x\eps \text{ is half-integer} \\
2 & \frac x\eps \text{ is integer}.
\end{cases}
\]
The atomistic system contains $\#(\calM)=2^{14}$ atoms.

The initial displacement has to conform with the assumption of absence of fast vibrations of the microstructure.
It is chosen in the following way:
First, we compute an equilibrium displacement $u$; i.e., such that $\<\delE(u), v\>_\calM=0$ for all $v\in\calU_{\per}(\calM)$.
Second, we compute an eigenvector of $\ddelE(u)$, $u_1$, corresponding to the mode oscillating most slowly.
Then, the initial displacement is taken to be $u^0 = u + 0.01\,\frac{u_1}{\|D u_1\|_{L^{\infty}}}$.
With such an initial displacement, the solution remains smooth (i.e., most of energy of the solution is contained in long wavelength modes) for times comparable to the largest oscillation period, and one can compare a QC approximation of the solution with the exact solution.
Beyond this critical time, the shock waves appear, which cause fast vibrations of the microstructure past them and hence make the approximation \eqref{eq:unsteady:homogenized} invalid.

We compare the reconstructed solution obtained by the HQC discretization in space with the reference solution obtained in the full atomistic computation.
The reconstruction of the HQC solution is performed similarly as described in Section \ref{HQC:nonlinear:reconstruction}.
The sampling domains $S_T^\rep$ were chosen to be $\eps\calP$ up to a shift in $\eps\bbZ$.
The HQC discretization is performed on a sequence of meshes with $h=\smfrac14,\smfrac18,\ldots$.
For the time integration, we use the Verlet method with the timestep $\tau = \frac{1}{20} h$ for the HQC solution and $\tau=\frac{1}{20} \eps$ for the reference atomistic solution.
We run the computation until $T=\frac{1}{20}$, which corresponds to about a quarter of a period of oscillation of the solution.

\begin{figure}
\begin{center}
$\mathstrut$
\hfill
\includegraphics{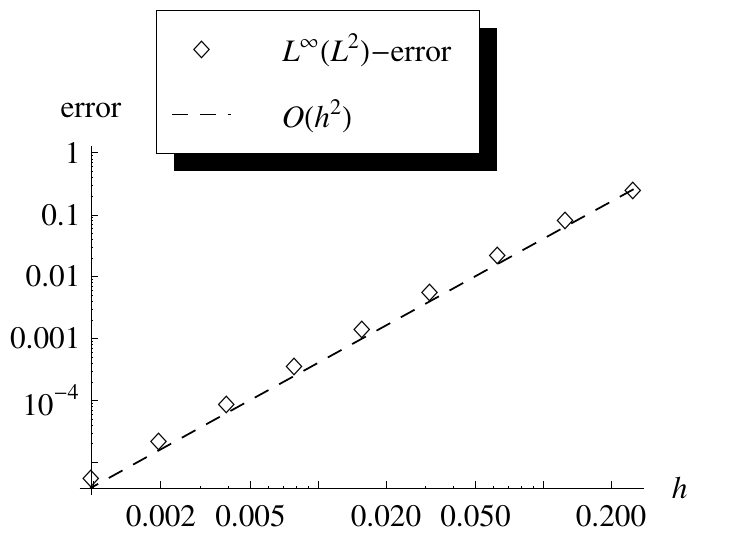}
\hfill
\includegraphics{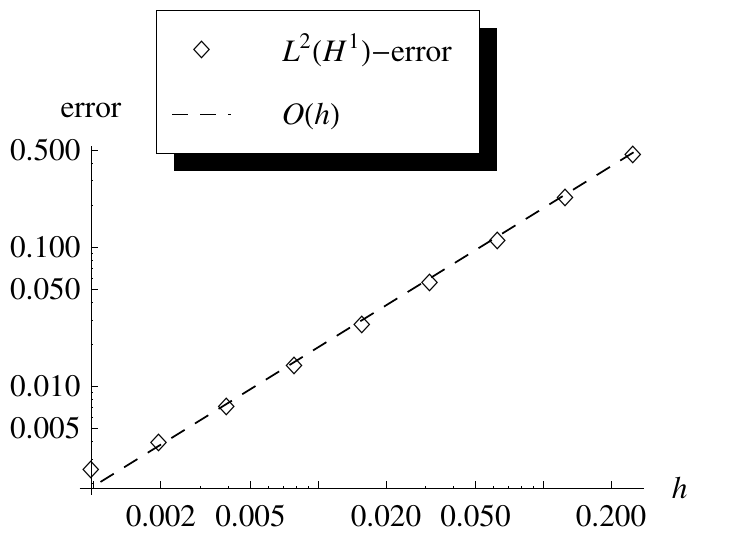}
\hfill
$\mathstrut$
\end{center}
\caption{Error of time-dependent solution in (the discrete analogues of) the $L^\infty([0,T]; L^2(\Omega))$-norm (left) and $L^2([0,T]; H^1(\Omega))$-norm (right).
}
\label{fig:dynamic_error}
\end{figure}

The errors in (the discrete analogues of) $L^\infty([0,T]; L^2(\Omega))$-norm and $L^2([0,T]; H^1(\Omega))$-norm are presented in Fig.~\ref{fig:dynamic_error}.
One can clearly observe for relatively large $h$ a second order convergence in the $L^2(\Omega)$-norm and a first order convergence in the $H^1(\Omega)$-norm, and the convergence seems to stagnate as $h$ is further reduced.

\section{Summary and Concluding Remarks}\label{sec:conclusion}

We have considered the problem of equilibrium of multilattice crystalline materials and discussed the application of the (local) QC method \cite{TadmorSmithBernsteinEtAl1999} for such materials.
We then have proposed a homogenization framework and, based on it, proposed a numerical macro-to-micro method which we called HQC.
We have shown that the three methods, namely the HQC method, the QC method applied to the discretely homogenized equations, and the multilattice QC, are equivalent.

Despite equivalence of the methods for statics of multilattice, we argue that the homogenization framework developed in this paper has several advantages.
First, it contributes to a better understanding of the multilattice QC method and provides a link to the existing theory of homogenization of PDEs.
In particular, we have generalized and applied the HQC method to the case of random materials and to the unsteady case, numerically demonstrating convergence of the proposed numerical method.
Second, the developed homogenization framework allows for application of analytical techniques available in the homogenization theory and thus seems most promising for convergence analysis of numerical methods for multilattices.
We refer to our preprint \cite{AbdulleLinShapeev2010} and an ongoing work \cite{AbdulleLinShapeev2011_analysis} for an example of such analysis.
We also note that the extension of the homogenization technique proposed in this paper to atomistic materials at finite temperature is of high interest.

\section*{Acknowledgments}

We thank the three anonymous referees for many comments that led to significant improvement of this paper.

\appendix

\section{Notations}\label{sec:notations}

In this appendix we gather the frequently used notations.

\subsection{Function spaces}\label{sec:notations:spaces}

For any finite set $S\subset \bbR^d$, we define the discrete averaging (integration) operator $\<\bullet\>_S$ by
\[
\<u\>_S := \frac{1}{\#(S)}\sum_{x\in S} u(x),
\]
and sometimes, more verbosely, as $\<u(x)\>_{x\in S}$.
Here $\#(S)$ is the number of elements in the set $S$.

We consider discrete periodic functions (e.g., displacements or external forces) with the periodic cell $\Omega = [0,1)^d$ ($d\in\bbN$), and the lattice (being, actually, the discrete periodic cell) $S \subset \Omega$ ($S=\calL, \calM$) containing a finite number of points: $\#(S)<\infty$.
The periodic extension of the lattice is denoted by $S_\per = S+\bbZ^d$.
Such space of periodic functions is denoted by
\begin{equation}\label{eq:Uper}
\calU_\per(S)= \big\{u: S_\per\to\bbR: \ u(x+a) = u(x)~\forall x\in S,~\forall a\in\bbZ^d \big\},
\end{equation}
and the space of periodic functions with zero average by
\[
\calU_{\#}(S)= \big\{u\in \calU_\per(S):~\<u\>_S=0\big\}.
\]
We do not have separate notations for scalar and vector-valued functions and explicitly state whether the function is scalar or vector-valued when it may cause ambiguity.

Similarly to the discrete averaging, we also use continuum averaging notation $\<u\>_\Omega := \int_\Omega u(x) \dx$, and for functions of two variables we write $\<v\>_{S_1\times S_2} := \big\<\<v\>_{S_2}\big\>_{S_1}$, where each $S_i$ ($i=1,2$) can be either continuous or discrete.

For vector-valued $u=u(x)$ and $v=v(x)$ we denote the pointwise scalar product by $u\cdot v$ (i.e., $(u\cdot v)(x) = u(x)\cdot v(x)$)
and the semi-inner product in $\calU_\per(\calL)$ by
\[
\<u,v\>_\calL = \<u\cdot v\>_\calL = \frac{1}{\#(\calL)}\sum_{x\in\calL} u(x)\cdot v(x).
\]
(It is a proper inner product only in $\calU_\#(\calL)$.)
We similarly define the pointwise scalar product and the (semi-)inner product for functions of continuum variable and for functions of several continuum or discrete variables.

\subsection{Operators}\label{sec:notations:operators}

For $u\,:\,S\to \bbR^d$ ($S=\calL, \calM$) we introduce the finite difference $D_{x,r} u$
\[
D_{x,r} u(x) := \frac{u(x + \eps r)-u(x)}{\eps}
\qquad(\text{for }x\in S,~r\in\bbR^d \text{~~such that }x+\eps r\in S)
.
\]
In addition to differentiation operators, we define for $u\in \calU_\per(\calL_1)$, the translation operator
$T_x u\in \calU_\per(\calL_1)$
\[
T_{x,r} u(x) := u(x+\eps r)
\qquad(\text{for }x\in S,~r\in\bbR^d \text{~~such that }x+\eps r\in S)
.
\]

The definitions of the discrete derivative and translation generalize to functions of two variables by considering the partial discrete derivative and translation operators; i.e.,
$D_{x,r},T_{x,r}$ applied to $u(\bullet,y)$ and
$D_{y,r},T_{y,r}$ applied to $u(x,\bullet)$.

In homogenization we consider ``traces on diagonal'' of functions of two variables, $v=v(x,\smfrac x\eps)$.
For such functions we introduce full translation and full derivative operators $T_r := T_{x,r} T_{y,r}$, $D_r := \smfrac1\eps (T_r-I)$ so that
\begin{equation}
\label{eq:full_derivative_relation}
(T_r u)|_{y=\smfrac x\eps} = T_{x,r} \Big(u|_{y=\smfrac x\eps}\Big),
\quad
\text{and}
\qquad
(D_r u)|_{y=\smfrac x\eps} = D_{x,r} \Big(u|_{y=\smfrac x\eps}\Big)
.
\end{equation}
The following relates the partial and the full derivatives:
\begin{align} \notag
D_r
=~& \smfrac1\eps (T_{x,r} T_{y,r}-I)
\\ =~& \notag
\smfrac1\eps (T_{x,r} T_{y,r}-T_{y,r}) + \smfrac1\eps (T_{y,r} - I)
\\ =~& \label{eq:full_and_partial}
D_{x,r} T_{y,r} + \smfrac 1\eps D_{y,r}
.
\end{align}

Notice that the variables $x$ and $y$ are not symmetric in the definition of full derivative.
If a function does not depend on $y$ then the full derivative coincides with the derivative in $x$ (likewise for the translation).
Hence, for functions of $x$ only, we sometimes omit the subscript $x$ in the operators $D_{x,r}$ and $T_{x,r}$.

For continuous functions we denote $\nabla u$ a gradient of $u$ and $\nabla_r u=(\nabla u)\cdot r$ a directional derivative.
For a vector-valued function $u$, the directional derivative, $\nabla_r u$ is defined componentwise and the gradient $\nabla u$ is a matrix such that $\nabla_r u = (\nabla u)r$.

\subsection{Functions of Vector-indexed Variables}\label{sec:notations:vector_indexed}

We consider a general form of interaction, where the energy of each atom depends arbitrarily on relative displacements of all the nearby atoms.
Namely, for the ``interaction neighborhood'' $\calR=\{r_1,\ldots,r_k\}$ we consider functions
\[
V(D_{r_1} u, D_{r_2} u, \ldots, D_{r_k} u).
\]
Since the interaction neighborhood may be different for different atoms (recall that we consider multilattices) and contain different number of neighbors $k$, we index derivatives directly with $r\in\calR$.
That is, we use the following notation for tuples $\alpha$ indexed with $r\in\calR$:
\[
(\alpha_r)_{r\in\calR} := (\alpha_{r_1},\ldots,\alpha_{r_k})
\quad \text{for } \calR=\{r_1,\ldots,r_k\}
\]
and define
\[
D_\calR u := (D_r u)_{r\in\calR},
\qquad
\nabla_\calR u := (\nabla_r u)_{r\in\calR}.
\]
Thus, for the functions of $\calR$-indexed tuples we write
\[
V(D_\calR u) := V(D_{r_1} u, D_{r_2} u, \ldots, D_{r_k} u).
\]

The common algebraic operations on $\calR$-indexed tuples are taken componentwise, e.g.:
\begin{equation}
\label{eq:vector_indexed}
D_\calR u + D_\calR v = (D_r u+D_r v)_{r\in\calR},
\qquad
\mF \calR = (\mF r)_{r\in\calR}
\quad\text{ etc.},
\end{equation}
which is fully analogous to the algebraic operations on $k$-dimensional vectors.

A partial derivative of $V(D_\calR u)$ with respect to $D_r u$ ($r\in\calR$) is denoted by $V'_r(D_\calR u)$.

\bibliographystyle{siam}
\bibliography{atm-hmg}

\begin{thebibliography}{10}

\bibitem{Abdulle2009}
{\sc A.~Abdulle}, {\em The finite element heterogeneous multiscale method: a
  computational strategy for multiscale {PDEs}}, GAKUTO Internat. Ser. Math.
  Sci. Appl., 31 (2009), pp.~135--184.

\bibitem{Abd11b}
\leavevmode\vrule height 2pt depth -1.6pt width 23pt, {\em A priori and a
  posteriori error analysis for numerical homogenization: a unified framework},
  Ser. Contemp. Appl. Math. CAM, 16 (2011), pp.~280--305.

\bibitem{AbdulleGrote2011}
{\sc A.~Abdulle and M.~J. Grote}, {\em Finite element heterogeneous multiscale
  method for the wave equation}, Multiscale Modeling \& Simulation, 9 (2011),
  pp.~766--792.

\bibitem{AbdulleLinShapeev2010}
{\sc A.~Abdulle, P.~Lin, and A.~V. Shapeev}, {\em Homogenization-based analysis
  of quasicontinuum method for complex crystals}, 2010.
\newblock arXiv:1006.0378v1.

\bibitem{AbdulleLinShapeev2011_analysis}
{\sc A.~Abdulle, P.~Lin, and A.~V. Shapeev}, {\em A one-dimensional nonlinear
  analysis of the multilattice quasicontinuum method}, in preparation.

\bibitem{AlicandroCicaleseGloria2011}
{\sc R.~Alicandro, M.~Cicalese, and A.~Gloria}, {\em Integral representation
  results for energies defined on stochastic lattices and application to
  nonlinear elasticity}, Arch. Ration. Mech. Anal., 200 (2011), pp.~881--943.

\bibitem{Babuska1976}
{\sc I.~Babuska}, {\em Homogenization and its application. {M}athematical and
  computational problems}, Numerical solution of partial differential
  equations,  (1976), pp.~89--115.

\bibitem{Bakhvalov1974}
{\sc N.~S. Bakhvalov}, {\em Averaged characteristics of bodies with a periodic
  structure}, Dokl. Akad. Nauk SSSR, 218 (1974), pp.~1046--1048.
\newblock English translation: Phys. Dokl. 19, 1974--1975.

\bibitem{BaumanOdenPrudhomme2009}
{\sc P.~T. Bauman, J.~T. Oden, and S.~Prudhomme}, {\em Adaptive multiscale
  modeling of polymeric materials with {Arlequin} coupling and {Goals}
  algorithms}, Comput. Methods Appl. Mech. Engrg., 198 (2009), pp.~799--818.

\bibitem{BensoussanLionsPapanicolaou1978}
{\sc A.~Bensoussan, J.-L. Lions, and G.~Papanicolaou}, {\em Asymptotic Analysis
  for Periodic Structure}, North Holland, Amersterdam, 1978.

\bibitem{BlancLeBrisLions2007a}
{\sc X.~Blanc, C.~Le~Bris, and P.~L. Lions}, {\em Atomistic to continuum limits
  for computational materials science}, Esaim-mathematical Modelling and
  Numerical Analysis-modelisation Mathematique Et Analyse Numerique, 41 (2007),
  pp.~391--426.

\bibitem{BlancLeLions2007_stochastic_lattices}
\leavevmode\vrule height 2pt depth -1.6pt width 23pt, {\em The energy of some
  microscopic stochastic lattices}, Archive For Rational Mechanics and
  Analysis, 184 (2007), pp.~303--339.

\bibitem{ChenFish2006}
{\sc W.~Chen and J.~Fish}, {\em A generalized space-time mathematical
  homogenization theory for bridging atomistic and continuum scales}, Internat.
  J. Numer. Methods Engrg., 67 (2006), pp.~253--271.

\bibitem{ChenFish2006a}
\leavevmode\vrule height 2pt depth -1.6pt width 23pt, {\em A mathematical
  homogenization perspective of virial stress}, Internat. J. Numer. Methods
  Engrg., 67 (2006), pp.~189--207.

\bibitem{Chung2004}
{\sc P.~Chung}, {\em Computational method for atomistic homogenization of
  nanopatterned point defect structures}, Int. J. Numer. Meth. Engng., 60
  (2004), pp.~833--859.

\bibitem{ChungNamburu2003}
{\sc P.~Chung and R.~Namburu}, {\em On a formulation for a multiscale
  atomistic-continuum homogenization method}, Internat. J. Solids Structures,
  40 (2003), pp.~2563--2588.

\bibitem{DobsonElliottLuskinEtAl2007}
{\sc M.~Dobson, R.~S. Elliott, M.~Luskin, and E.~B. Tadmor}, {\em A
  multilattice quasicontinuum for phase transforming materials: Cascading
  {Cauchy Born} kinematics}, Journal of Computer-Aided Materials Design, 14
  (2007), pp.~219--237.

\bibitem{DobsonLuskinOrtner2010b}
{\sc M.~Dobson, M.~Luskin, and C.~Ortner}, {\em Stability, instability, and
  error of the force-based quasicontinuum approximation}, Archive for Rational
  Mechanics and Analysis, 197 (2010), pp.~179--202.

\bibitem{EEL2007}
{\sc W.~E, B.~Engquist, X.~Li, W.~Ren, and E.~Vanden-Eijnden}, {\em
  Heterogeneous multiscale methods: a review}, Commun. Comput. Phys., 2 (2007),
  pp.~367--450.

\bibitem{EMing2007}
{\sc W.~E and P.~Ming}, {\em {Cauchy-Born} rule and the stability of
  crystalline solids: Static problems}, Arch. Ration. Mech. Anal., 183 (2007),
  pp.~241--297.

\bibitem{EfendievHou2009}
{\sc Y.~Efendiev and T.~Y. Hou}, {\em Multiscale finite element methods},
  vol.~4 of Surveys and Tutorials in the Applied Mathematical Sciences,
  Springer, New York, 2009.
\newblock Theory and applications.

\bibitem{EngquistHolstRunborg2011}
{\sc B.~Engquist, H.~Holst, and O.~Runborg}, {\em Multi-scale methods for wave
  propagation in heterogeneous media}, Commun. Math. Sci., 9 (2011),
  pp.~33--56.

\bibitem{Ericksen2008}
{\sc J.~L. Ericksen}, {\em On the {Cauchy-Born} rule}, Math. Mech. Solids, 13
  (2008), pp.~199--220.

\bibitem{FishChenLi2007}
{\sc J.~Fish, W.~Chen, and R.~Li}, {\em Generalized mathematical homogenization
  of atomistic media at finite temperatures in three dimensions}, Comput.
  Methods Appl. Mech. Engrg., 196 (2007), pp.~908--922.

\bibitem{FrieseckeTheil2002}
{\sc G.~Friesecke and F.~Theil}, {\em Validity and failure of the {Cauchy-Born}
  hypothesis in a two-dimensional mass-spring lattice}, J. Nonlinear Sci., 12
  (2002), pp.~445--478.

\bibitem{GeersKouznetsovaBrekelmans2010}
{\sc M.~Geers, V.~Kouznetsova, and W.~Brekelmans}, {\em Multi-scale
  computational homogenization: Trends and challenges}, J. Comput. Appl. Math.,
   (2010).
\newblock In press.

\bibitem{GloriaOtto2011_variance_estimate}
{\sc A.~Gloria and F.~Otto}, {\em An optimal variance estimate in stochastic
  homogenization of discrete elliptic equations}, Annals of Probability, 39
  (2011), pp.~779--856.

\bibitem{HudsonOrtner2012}
{\sc T.~Hudson and C.~Ortner}, {\em On the stability of {B}ravais lattices and
  their {C}auchy--{B}orn approximations}, M2AN Math. Model. Numer. Anal., 46
  (2012), pp.~81--110.

\bibitem{Lin2003}
{\sc P.~Lin}, {\em The theoretical and numerical analysis of the quasicontinuum
  approximation of a material particle model}, Math. of Comp., 72 (2003),
  pp.~657--675.

\bibitem{Lin2007}
\leavevmode\vrule height 2pt depth -1.6pt width 23pt, {\em Convergence analysis
  of a quasi-continuum approximation for a two-dimensional material without
  defects}, SIAM J. Numer. Anal., 45 (2007), pp.~313--332.

\bibitem{LuMing}
{\sc J.~Lu and P.~Ming}, {\em Convergence of a force-based hybrid method for
  atomistic and continuum models in three dimension}.
\newblock arXiv:1102.2523.

\bibitem{MakridakisSuli}
{\sc C.~Makridakis and E.~S\"uli}, {\em Finite element analysis of cauchy-born
  approximations to atomistic models}.
\newblock Eprints Archive of the Mathematical Institute, University of Oxford,
  ID code: 1451, 2011.

\bibitem{MieheBayreuther2007}
{\sc C.~Miehe and C.~G. Bayreuther}, {\em On multiscale {FE} analyses of
  heterogeneous structures: {F}rom homogenization to multigrid solvers},
  Internat. J. Numer. Methods Engrg., 71 (2007), pp.~1135--1180.

\bibitem{MillerTadmor2002}
{\sc R.~E. Miller and E.~B. Tadmor}, {\em The quasicontinuum method:
  {O}verview, applications and current directions}, Journal of Computer-Aided
  Materials Design, 9 (2002), pp.~203--239.

\bibitem{MingYang2009}
{\sc P.~B. Ming and J.~Z. Yang}, {\em Analysis of a one-dimensional nonlocal
  quasi-continuum method}, Multiscale Model. Simul., 7 (2009), pp.~1838--1875.

\bibitem{OrtnerShapeev2011}
{\sc C.~Ortner and A.~V. Shapeev}, {\em Analysis of an energy-based
  quasicontinuum approximation of a vacancy in the {2D} hexagonal lattice}.
\newblock In preparation.

\bibitem{OrtnerSuli2008}
{\sc C.~Ortner and E.~S{\"u}li}, {\em Analysis of a quasicontinuum method in
  one dimension}, M2AN Math. Model. Numer. Anal., 42 (2008), pp.~57--91.

\bibitem{S'anchez-Palencia1980}
{\sc E.~S\'anchez-Palencia}, {\em Non-homogeneous media and vibration theory},
  Springer-Verlag, 1980.

\bibitem{SorkinElliottTadmor}
{\sc V.~Sorkin, R.~S. Elliott, and E.~B. Tadmor}, {\em A local quasicontinuum
  for 3{D} multilattice crystalline materials: {A}pplication to shape-memory
  alloys}.
\newblock manuscript.

\bibitem{Stakgold1950}
{\sc I.~Stakgold}, {\em The {Cauchy} relations in a molecular theory of
  elasticity}, Quarterly of Applied Mechanics, 8 (1950), pp.~169--186.

\bibitem{TadmorPhillipsOrtiz1996}
{\sc E.~Tadmor, R.~Phillips, and M.~Ortiz}, {\em Quasicontinuum analysis of
  defects in solids}, Philos. Mag. A, 73 (1996), pp.~1529--1563.

\bibitem{TadmorSmithBernsteinEtAl1999}
{\sc E.~Tadmor, G.~Smith, N.~Bernstein, and E.~Kaxiras}, {\em Mixed finite
  element and atomistic formulation for complex crystals}, Phys. Rev. B, 59
  (1999), pp.~235--245.

\bibitem{TeradaKikuchi2001}
{\sc K.~Terada and N.~Kikuchi}, {\em A class of general algorithms for
  multi-scale analyses of heterogeneous media}, Comput. Methods Appl. Mech.
  Engrg., 190 (2001), pp.~5427--5464.

\bibitem{Tor005}
{\sc S.~Torquato and S.~Torquato}, {\em Random Heterogeneous Materials},
  vol.~16, Springer, IAM, 2005.

\bibitem{VanKotenLiLuskinEtAl2012}
{\sc B.~Van~Koten, X.~H. Li, M.~Luskin, and C.~Ortner}, {\em A computational
  and theoretical investigation of the accuracy of quasicontinuum methods}, in
  Numerical Analysis of Multiscale Problems, I.~Graham, T.~Hou, O.~Lakkis, and
  R.~Scheichl, eds., Springer Lecture Notes in Computational Science and
  Engineering 83, 2012.

\end{thebibliography}

\end{document}